\documentclass[a4paper]{article}
\usepackage[latin1]{inputenc}
\usepackage{fontenc,dsfont}
\usepackage{amsmath,a4}
\usepackage{amsfonts}
\usepackage[english]{babel}
\usepackage{amssymb}
\usepackage{latexsym}
\usepackage{paralist}

\newtheorem{theorem}{Theorem}[section]
\newtheorem{proposition}{Proposition}[section]
\newtheorem{remark}{Remark}[section]

\newenvironment{proof}{
\noindent{\bf Proof:}}{\hspace*{\fill}$\blacksquare$ \newline}
\begin{document}

\begin{center}
{\normalsize \textbf{Best Approximation Pair of Two Linear Varieties via an
(In)Equality by (Fan-Todd) Beesack}} \\[12mm]
\textsc{M. A. Facas Vicente$^{\ast }$$^{1}$$^{2}$, Fernando Martins$^{3}$$%
^{4}$, Cecília Costa$^{5}$$^{6}$ and José Vitória$^{1}$} \\[8mm]
%     {\small  Submitted March, 2011 }\\[5mm]

\begin{minipage}{123mm}
{\small {\sc Abstract.}
The closest point of a linear variety to an external point is found by using
the equality case of an Ostrowski's type inequality. This point is given in
closed form as the quotient of a (formal) and a (scalar) Gram determinant.
Then, the best approximation pair of points onto two linear varieties is
given, as well as characterization of this pair of best approximation points.}
\end{minipage}
\end{center}

\renewcommand{\thefootnote}{} \footnotetext{$^{\ast }$\thinspace
Corresponding author.} \footnotetext{%
}
 \footnotetext{$^{1}$\thinspace Department of Mathematics,
Faculty of Sciences and Technology, University of Coimbra, Apartado 3008,
3001-454 Coimbra, Portugal. E-mails: \textit{vicente@mat.uc.pt} (M. A. Facas
Vicente) \textit{jvitoria@mat.uc.pt} (José Vitória).} \footnotetext{$^{2}$%
\thinspace Supported by INESC-C --- Instituto de Engenharia de Sistemas e
Computadores-Coimbra, Rua Antero de Quental, 199, 3000-033 Coimbra, Portugal.%
} \footnotetext{$^{3}$\thinspace Coimbra College of Education, Polytechnic
Institute of Coimbra, Praça Heróis do Ultramar, Solum, 3030-329 Coimbra,
Portugal. E-mail: \textit{fmlmartins@esec.pt} (Fernando Martins).}
\footnotetext{$^{4}$\thinspace Supported by Instituto de Telecomunicações, Pó%
lo de Coimbra, Delegação da Covilhã, Portugal.} \footnotetext{$^{5}$%
\thinspace Department of Mathematics and CM-UTAD, University of Trá%
s-os-Montes e Alto Douro, Apartado 1013, 5001-801 Vila Real, Portugal.
E-mail: \textit{mcosta@utad.pt} (Cecília Costa).} \footnotetext{$^{6}$%
\thinspace Centro de Investigação e Desenvolvimento Matemática e Aplicações
da Universidade de Aveiro, University of Aveiro, 3810-193 Aveiro, Portugal.}

\section{Introduction}

In this paper, we answer an implicit open question by Ky Fan and John Todd
\cite[page 63]{FanTodd}. We give a determinantal formula for the point where
the inequality of the above referred to authors turns into equality, thusly
obtaining the point of least norm of the intersection of certain
hyperplanes. We present a result, in terms of Gram determinants, for the
minimum distance from a certain linear variety to the origin of coordinates
[Proposition \ref{prop21}]. We note that this formula generalizes the one
Mitrinovic \cite{Mitrinovic, Mitrinovic2} has given in the case of two
equations. This best approximation problem was dealt with in \cite{Vitoria},
where the centre of (degenerate) hyperquadrics plays a decisive rôle. In
\cite{Vitoria}, no answer in closed form was given.

In this paper, we give a new proof of Beesack's inequality (\cite[Theorem 1]%
{Beesack}; \cite[Theorem 1.7]{Varosanec}), by following arguments used in
\cite[page 63, Lemma]{FanTodd}.

The Beesack's formula [Theorem \ref{PBeesack}] gives the point of a general
linear variety closest to the origin of the coordinates. We extend the
formula of Beesack \cite[Theorem 1]{Beesack} in order to get the nearest
point of a linear variety to an external point, in $\mathrm{I\kern-.17emR}%
^{n}$. When extending Theorem \ref{PBeesack}, we obtain the projection of an
external point onto a general linear variety [Proposition \ref{prop41}].
This Proposition \ref{prop41} is used for getting the best approximation
points of two linear varieties [Proposition \ref{prop51}]. Also a
characterizion of the best approximation pair of two linear varieties is
presented [Proposition 5.2].

Our context is the Eu\-cli\-dean space $\mathrm{I\kern-.17emR}^{n}$, endowed
with the standard unit basis
\begin{equation*}
\left( \overrightarrow{e}_{1},\overrightarrow{e}_{2},\ldots ,\overrightarrow{%
e}_{n}\right)
\end{equation*}%
and the ordinary inner product%
\begin{equation*}
\overrightarrow{u}\bullet \overrightarrow{v}=u_{1}v_{1}+u_{2}v_{2}+\cdots
+u_{n}v_{n},
\end{equation*}%
where%
\begin{equation*}
\overrightarrow{a}=a_{1}\overrightarrow{e}_{1}+a_{2}\overrightarrow{e}%
_{2}+\cdots +a_{n}\overrightarrow{e}_{n}=\left( a_{1},a_{2},\ldots
,a_{n}\right) .
\end{equation*}

The Eu\-cli\-dean norm $\left\Vert \overrightarrow{a}\right\Vert =+\sqrt{%
\overrightarrow{a}\bullet \overrightarrow{a}}$ is used and the Gram
determinant is%
\begin{equation}
G\left( \overrightarrow{p}_{1},\overrightarrow{p}_{2},\ldots ,%
\overrightarrow{p}_{r}\right) =\det \left[
\begin{array}{cccc}
\overrightarrow{p}_{1}\bullet \overrightarrow{p}_{1} & \overrightarrow{p}%
_{1}\bullet \overrightarrow{p}_{2} & \cdots & \overrightarrow{p}_{1}\bullet
\overrightarrow{p}_{r} \\
\overrightarrow{p}_{2}\bullet \overrightarrow{p}_{1} & \overrightarrow{p}%
_{2}\bullet \overrightarrow{p}_{2} & \cdots & \overrightarrow{p}_{2}\bullet
\overrightarrow{p}_{r} \\
\vdots & \vdots & \ddots & \vdots \\
\overrightarrow{p}_{r}\bullet \overrightarrow{p}_{1} & \overrightarrow{p}%
_{r}\bullet \overrightarrow{p}_{2} & \cdots & \overrightarrow{p}_{r}\bullet
\overrightarrow{p}_{r}%
\end{array}%
\right] ,\text{ }1\leq r\leq n.  \label{x01}
\end{equation}

It is well known that $G\left( \overrightarrow{p}_{1},\overrightarrow{p}%
_{2},\ldots ,\overrightarrow{p}_{r}\right) \geq 0$\ and $G\left(
\overrightarrow{p}_{1},\overrightarrow{p}_{2},\ldots ,\overrightarrow{p}%
_{r}\right) =0$\ if and only if the vectors $\overrightarrow{p}_{1},%
\overrightarrow{p}_{2},\ldots ,\overrightarrow{p}_{r}$ are linearly
dependent. See, for example, \cite[page 132]{Deutsch}.

Some abuse of notation, authorized by adequate isomorphisms, is to be
declared, notably the identification of point, vector, ordered set,
column-matrix.

This paper is organized in seven sections. In Section 2, we present and
prove a result, Proposition \ref{prop21}, which answers an open question of
Fan and Todd and make a remark concerning a formula of Mitrinovic. Section 3
is dedicated to a generalization of Proposition \ref{prop21}, this meaning
that we study the projection of the origin onto a general linear variety. In
Section 4, we deal with the projection of an external point onto a general
linear variety. In the Section 5, we treat the distance between two disjoint
linear varieties. We get and characterize the best two points, one on each
linear variety, that are the extremities of the straight line segment that
materializes the distance between the two linear varieties. An illustrative
numerical example is presented in Section 6. Finally, in Section 7, we draw
some conclusions.

\section{The minimum norm vector of a certain linear variety}

In this section, we state the Proposition 2.1, which solves an old open
question of Fan and Todd. The proof makes use of a result of the mentioned
authors.

The next result \cite[page 63, Lemma]{FanTodd} gives the radius of the
sphere tangent to a certain linear variety, as the quotient of two Gram
determinants.

\begin{theorem}
Let $\overrightarrow{a}_{1},\overrightarrow{a}_{2},\ldots ,\overrightarrow{a}%
_{m}$ be $m$ linearly independent vectors in $\mathrm{I\kern-.17emR}^{n}$, $%
2\leq m\leq n$. If a vector $\overrightarrow{x}\in \mathrm{I\kern-.17emR}%
^{n} $ varies under the conditions%
\begin{equation}
\begin{array}{l}
\overrightarrow{a}_{i}\bullet \overrightarrow{x}=0,\qquad \text{with }1\leq
i\leq m-1 \\
\overrightarrow{a}_{m}\bullet \overrightarrow{x}=1,%
\end{array}
\label{eq09}
\end{equation}%
then%
\begin{equation}
\overrightarrow{x}\bullet \overrightarrow{x}\geq \frac{G\left(
\overrightarrow{a}_{1},\overrightarrow{a}_{2},\ldots ,\overrightarrow{a}%
_{m-1}\right) }{G\left( \overrightarrow{a}_{1},\overrightarrow{a}_{2},\ldots
,\overrightarrow{a}_{m-1},\overrightarrow{a}_{m}\right) }.  \label{eq05}
\end{equation}%
Furthermore, the minimum value is obtained if and only if $\overrightarrow{x}
$ is a linear combination of $\overrightarrow{a}_{1},\overrightarrow{a}%
_{2},\ldots ,\overrightarrow{a}_{m}$.
\end{theorem}

For the sake of completeness and for later use in the proof of our
Proposition 2.1., we present here, essentially, the proof given by Fan and
Todd \cite[page 63, Lemma]{FanTodd}.

\begin{proof}
For the vector $\overrightarrow{x}$ satisfying conditions (\ref{eq09}), we
have%
\begin{equation}
G\left( \overrightarrow{a}_{1},\overrightarrow{a}_{2},\ldots ,%
\overrightarrow{a}_{m},\overrightarrow{x}\right) =-G\left( \overrightarrow{a}%
_{1},\overrightarrow{a}_{2},\ldots ,\overrightarrow{a}_{m-1}\right) +\left(
\overrightarrow{x}\bullet \overrightarrow{x}\right) G\left( \overrightarrow{a%
}_{1},\overrightarrow{a}_{2},\ldots ,\overrightarrow{a}_{m}\right) \geq 0.
\notag
\end{equation}%
Hence%
\begin{equation*}
\overrightarrow{x}\bullet \overrightarrow{x}\geq \frac{G\left(
\overrightarrow{a}_{1},\overrightarrow{a}_{2},\ldots ,\overrightarrow{a}%
_{m-1}\right) }{G\left( \overrightarrow{a}_{1},\overrightarrow{a}_{2},\ldots
,\overrightarrow{a}_{m-1},\overrightarrow{a}_{m}\right) }.
\end{equation*}%
By hypothesis, the vectors $\overrightarrow{a}_{1},\overrightarrow{a}%
_{2},\ldots ,\overrightarrow{a}_{m}$ are linearly independent, so
\begin{equation*}
G\left( \overrightarrow{a}_{1},\overrightarrow{a}_{2},\ldots ,%
\overrightarrow{a}_{m},\overrightarrow{x}\right) =0
\end{equation*}%
if and only if $\overrightarrow{x}$\ is a linear combination of $%
\overrightarrow{a}_{1},\overrightarrow{a}_{2},\ldots ,\overrightarrow{a}_{m}$%
. It follows that
\begin{equation}
\overrightarrow{x}\bullet \overrightarrow{x}=\frac{G\left( \overrightarrow{a}%
_{1},\overrightarrow{a}_{2},\ldots ,\overrightarrow{a}_{m-1}\right) }{%
G\left( \overrightarrow{a}_{1},\overrightarrow{a}_{2},\ldots ,%
\overrightarrow{a}_{m-1},\overrightarrow{a}_{m}\right) }  \label{eq04}
\end{equation}%
if and only if the vector $\overrightarrow{x}$\ is of the form $%
\overrightarrow{x}=\alpha _{1}\overrightarrow{a}_{1}+\alpha _{2}%
\overrightarrow{a}_{2}+\cdots +\alpha _{m}\overrightarrow{a}_{m}$.
\end{proof}

Now we are in a position for stating the equality case. A determinantal
formula for the closest vector to the origin lying in a certain linear
variety is given.

\begin{proposition}
\label{prop21}

\begin{enumerate}
\item Let $\overrightarrow{a}_{1},\overrightarrow{a}_{2},\ldots ,%
\overrightarrow{a}_{m}$ be linearly independent vectors in $\mathrm{I\kern%
-.17emR}^{n}$. The minimum Eu\-cli\-dean norm vector in $\mathrm{I\kern%
-.17emR}^{n}$ satisfying the equations%
\begin{equation}
\begin{array}{l}
\overrightarrow{a}_{1}\bullet \overrightarrow{x}=0 \\
\overrightarrow{a}_{2}\bullet \overrightarrow{x}=0 \\
\text{ \ \ \ \ \ \ \ }\vdots \\
\overrightarrow{a}_{m-1}\bullet \overrightarrow{x}=0 \\
\overrightarrow{a}_{m}\bullet \overrightarrow{x}=1%
\end{array}
\label{eq03}
\end{equation}%
is given by%
\begin{equation}
\overrightarrow{s}=\frac{\left\vert
\begin{array}{ccccc}
\overrightarrow{a}_{1}\bullet \overrightarrow{a}_{1} & \overrightarrow{a}%
_{1}\bullet \overrightarrow{a}_{2} & \cdots & \overrightarrow{a}_{1}\bullet
\overrightarrow{a}_{m-1} & \overrightarrow{a}_{1}\bullet \overrightarrow{a}%
_{m} \\
\vdots & \vdots & \ddots & \vdots & \vdots \\
\overrightarrow{a}_{m-1}\bullet \overrightarrow{a}_{1} & \overrightarrow{a}%
_{m-1}\bullet \overrightarrow{a}_{2} & \cdots & \overrightarrow{a}%
_{m-1}\bullet \overrightarrow{a}_{m-1} & \overrightarrow{a}_{m-1}\bullet
\overrightarrow{a}_{m} \\
\overrightarrow{a}_{1} & \overrightarrow{a}_{2} & \cdots & \overrightarrow{a}%
_{m-1} & \overrightarrow{a}_{m}%
\end{array}%
\right\vert }{\left\vert
\begin{array}{ccccc}
\overrightarrow{a}_{1}\bullet \overrightarrow{a}_{1} & \overrightarrow{a}%
_{1}\bullet \overrightarrow{a}_{2} & \cdots & \overrightarrow{a}_{1}\bullet
\overrightarrow{a}_{m-1} & \overrightarrow{a}_{1}\bullet \overrightarrow{a}%
_{m} \\
\overrightarrow{a}_{2}\bullet \overrightarrow{a}_{1} & \overrightarrow{a}%
_{2}\bullet \overrightarrow{a}_{2} & \cdots & \overrightarrow{a}_{2}\bullet
\overrightarrow{a}_{m-1} & \overrightarrow{a}_{2}\bullet \overrightarrow{a}%
_{m} \\
\vdots & \vdots & \ddots & \vdots & \vdots \\
\overrightarrow{a}_{m-1}\bullet \overrightarrow{a}_{1} & \overrightarrow{a}%
_{m-1}\bullet \overrightarrow{a}_{2} & \cdots & \overrightarrow{a}%
_{m-1}\bullet \overrightarrow{a}_{m-1} & \overrightarrow{a}_{m-1}\bullet
\overrightarrow{a}_{m} \\
\overrightarrow{a}_{m}\bullet \overrightarrow{a}_{1} & \overrightarrow{a}%
_{m}\bullet \overrightarrow{a}_{2} & \cdots & \overrightarrow{a}_{m}\bullet
\overrightarrow{a}_{m-1} & \overrightarrow{a}_{m}\bullet \overrightarrow{a}%
_{m}%
\end{array}%
\right\vert },  \label{eq02}
\end{equation}%
where the determinant in the numerator is to be expanded by the last row, in
order to yield a linear combination of the vectors $\overrightarrow{a}_{1},%
\overrightarrow{a}_{2},\ldots ,\overrightarrow{a}_{m}$.

\item Furthermore,%
\begin{equation*}
\left\Vert \overrightarrow{s}\right\Vert ^{2}=\overrightarrow{s}\bullet
\overrightarrow{s}=\frac{G\left( \overrightarrow{a}_{1},\overrightarrow{a}%
_{2},\ldots ,\overrightarrow{a}_{m-1}\right) }{G\left( \overrightarrow{a}%
_{1},\overrightarrow{a}_{2},\ldots ,\overrightarrow{a}_{m-1},\overrightarrow{%
a}_{m}\right) }.
\end{equation*}
\end{enumerate}
\end{proposition}

\begin{proof}
Part 1. We look for the scalars $\alpha _{1},\alpha _{2},\ldots ,\alpha _{m}$%
, such that the vector $\overrightarrow{x}=\alpha _{1}\overrightarrow{a}%
_{1}+\alpha _{2}\overrightarrow{a}_{2}+\cdots +\alpha _{m}\overrightarrow{a}%
_{m}$ satisfies the conditions (\ref{eq03}).

For that end, we solve the system%
\begin{equation*}
\left[ \!\!\!%
\begin{array}{ccccc}
\overrightarrow{a}_{1}\bullet \overrightarrow{a}_{1} & \overrightarrow{a}%
_{1}\bullet \overrightarrow{a}_{2} & \cdots & \overrightarrow{a}_{1}\bullet
\overrightarrow{a}_{m-1} & \overrightarrow{a}_{1}\bullet \overrightarrow{a}%
_{m} \\
\overrightarrow{a}_{2}\bullet \overrightarrow{a}_{1} & \overrightarrow{a}%
_{2}\bullet \overrightarrow{a}_{2} & \cdots & \overrightarrow{a}_{2}\bullet
\overrightarrow{a}_{m-1} & \overrightarrow{a}_{2}\bullet \overrightarrow{a}%
_{m} \\
\vdots & \vdots & \ddots & \vdots & \vdots \\
\overrightarrow{a}_{m-1}\bullet \overrightarrow{a}_{1} & \overrightarrow{a}%
_{m-1}\bullet \overrightarrow{a}_{2} & \cdots & \overrightarrow{a}%
_{m-1}\bullet \overrightarrow{a}_{m-1} & \overrightarrow{a}_{m-1}\bullet
\overrightarrow{a}_{m} \\
\overrightarrow{a}_{m}\bullet \overrightarrow{a}_{1} & \overrightarrow{a}%
_{m}\bullet \overrightarrow{a}_{2} & \cdots & \overrightarrow{a}_{m}\bullet
\overrightarrow{a}_{m-1} & \overrightarrow{a}_{m}\bullet \overrightarrow{a}%
_{m}%
\end{array}%
\!\!\!\right] \left[ \!\!\!%
\begin{array}{c}
\alpha _{1} \\
\alpha _{2} \\
\vdots \\
\alpha _{m-1} \\
\alpha _{m}%
\end{array}%
\!\!\!\right] =\left[ \!\!\!%
\begin{array}{c}
0 \\
0 \\
\vdots \\
0 \\
1%
\end{array}%
\!\!\!\right] .
\end{equation*}

As the vectors $\overrightarrow{a}_{1},\overrightarrow{a}_{2},\ldots ,%
\overrightarrow{a}_{m}$ are, by hypothesis, linearly independent, the
determinant of the matrix of the above system, which is the Gram determinant%
\begin{equation*}
G\left( \overrightarrow{a}_{1},\overrightarrow{a}_{2},\ldots ,%
\overrightarrow{a}_{m-1},\overrightarrow{a}_{m}\right) ,
\end{equation*}
is non null.

So, by the Cramer's Rule, we have%
\begin{equation*}
\alpha _{i}=\frac{\left\vert
\begin{array}{ccccccc}
\overrightarrow{a}_{1}\bullet \overrightarrow{a}_{1} & \cdots &
\overrightarrow{a}_{1}\bullet \overrightarrow{a}_{i-1} & 0 & \overrightarrow{%
a}_{1}\bullet \overrightarrow{a}_{i+1} & \cdots & \overrightarrow{a}%
_{1}\bullet \overrightarrow{a}_{m} \\
\overrightarrow{a}_{2}\bullet \overrightarrow{a}_{1} & \cdots &
\overrightarrow{a}_{2}\bullet \overrightarrow{a}_{i-1} & 0 & \overrightarrow{%
a}_{2}\bullet \overrightarrow{a}_{i+1} & \cdots & \overrightarrow{a}%
_{2}\bullet \overrightarrow{a}_{m} \\
\vdots & \ddots & \vdots & \vdots & \vdots & \ddots & \vdots \\
\overrightarrow{a}_{m-1}\bullet \overrightarrow{a}_{1} & \cdots &
\overrightarrow{a}_{m-1}\bullet \overrightarrow{a}_{i-1} & 0 &
\overrightarrow{a}_{m-1}\bullet \overrightarrow{a}_{i+1} & \cdots &
\overrightarrow{a}_{m-1}\bullet \overrightarrow{a}_{m} \\
\overrightarrow{a}_{m}\bullet \overrightarrow{a}_{1} & \cdots &
\overrightarrow{a}_{m}\bullet \overrightarrow{a}_{i-1} & 1 & \overrightarrow{%
a}_{m}\bullet \overrightarrow{a}_{i+1} & \cdots & \overrightarrow{a}%
_{m}\bullet \overrightarrow{a}_{m}%
\end{array}%
\right\vert }{\left\vert
\begin{array}{ccccc}
\overrightarrow{a}_{1}\bullet \overrightarrow{a}_{1} & \overrightarrow{a}%
_{1}\bullet \overrightarrow{a}_{2} & \cdots & \overrightarrow{a}_{1}\bullet
\overrightarrow{a}_{m-1} & \overrightarrow{a}_{1}\bullet \overrightarrow{a}%
_{m} \\
\overrightarrow{a}_{2}\bullet \overrightarrow{a}_{1} & \overrightarrow{a}%
_{2}\bullet \overrightarrow{a}_{2} & \cdots & \overrightarrow{a}_{2}\bullet
\overrightarrow{a}_{m-1} & \overrightarrow{a}_{2}\bullet \overrightarrow{a}%
_{m} \\
\vdots & \vdots & \ddots & \vdots & \vdots \\
\overrightarrow{a}_{m-1}\bullet \overrightarrow{a}_{1} & \overrightarrow{a}%
_{m-1}\bullet \overrightarrow{a}_{2} & \cdots & \overrightarrow{a}%
_{m-1}\bullet \overrightarrow{a}_{m-1} & \overrightarrow{a}_{m-1}\bullet
\overrightarrow{a}_{m} \\
\overrightarrow{a}_{m}\bullet \overrightarrow{a}_{1} & \overrightarrow{a}%
_{m}\bullet \overrightarrow{a}_{2} & \cdots & \overrightarrow{a}_{m}\bullet
\overrightarrow{a}_{m-1} & \overrightarrow{a}_{m}\bullet \overrightarrow{a}%
_{m}%
\end{array}%
\right\vert },
\end{equation*}%
with $i=1,\ldots m$.

Here, for brevity, we introduce some notations:%
\begin{equation*}
\alpha _{i}=\frac{G_{i}}{G},\quad \alpha _{i}\overrightarrow{a}_{i}=\frac{%
G_{i}}{G}\overrightarrow{a}_{i}:=\frac{\overrightarrow{G}_{i}}{G},
\end{equation*}%
where
\begin{equation*}
G=G\left( \overrightarrow{a}_{1},\overrightarrow{a}_{2},\ldots ,%
\overrightarrow{a}_{m-1},\overrightarrow{a}_{m}\right),
\end{equation*}
\begin{equation*}
G_{i}=\left\vert
\begin{array}{ccccccc}
\overrightarrow{a}_{1}\bullet \overrightarrow{a}_{1} & \cdots &
\overrightarrow{a}_{1}\bullet \overrightarrow{a}_{i-1} & 0 & \overrightarrow{%
a}_{1}\bullet \overrightarrow{a}_{i+1} & \cdots & \overrightarrow{a}%
_{1}\bullet \overrightarrow{a}_{m} \\
\overrightarrow{a}_{2}\bullet \overrightarrow{a}_{1} & \cdots &
\overrightarrow{a}_{2}\bullet \overrightarrow{a}_{i-1} & 0 & \overrightarrow{%
a}_{2}\bullet \overrightarrow{a}_{i+1} & \cdots & \overrightarrow{a}%
_{2}\bullet \overrightarrow{a}_{m} \\
\vdots & \ddots & \vdots & \vdots & \vdots & \ddots & \vdots \\
\overrightarrow{a}_{m-1}\bullet \overrightarrow{a}_{1} & \cdots &
\overrightarrow{a}_{m-1}\bullet \overrightarrow{a}_{i-1} & 0 &
\overrightarrow{a}_{m-1}\bullet \overrightarrow{a}_{i+1} & \cdots &
\overrightarrow{a}_{m-1}\bullet \overrightarrow{a}_{m} \\
\overrightarrow{a}_{m}\bullet \overrightarrow{a}_{1} & \cdots &
\overrightarrow{a}_{m}\bullet \overrightarrow{a}_{i-1} & 1 & \overrightarrow{%
a}_{m}\bullet \overrightarrow{a}_{i+1} & \cdots & \overrightarrow{a}%
_{m}\bullet \overrightarrow{a}_{m}%
\end{array}%
\right\vert
\end{equation*}%
and the symbolic determinant%
\begin{equation*}
\overrightarrow{G}_{i}=\left\vert
\begin{array}{ccccccc}
\overrightarrow{a}_{1}\bullet \overrightarrow{a}_{1} & \cdots &
\overrightarrow{a}_{1}\bullet \overrightarrow{a}_{i-1} & 0 & \overrightarrow{%
a}_{1}\bullet \overrightarrow{a}_{i+1} & \cdots & \overrightarrow{a}%
_{1}\bullet \overrightarrow{a}_{m} \\
\overrightarrow{a}_{2}\bullet \overrightarrow{a}_{1} & \cdots &
\overrightarrow{a}_{2}\bullet \overrightarrow{a}_{i-1} & 0 & \overrightarrow{%
a}_{2}\bullet \overrightarrow{a}_{i+1} & \cdots & \overrightarrow{a}%
_{2}\bullet \overrightarrow{a}_{m} \\
\vdots & \ddots & \vdots & \vdots & \vdots & \ddots & \vdots \\
\overrightarrow{a}_{m-1}\bullet \overrightarrow{a}_{1} & \cdots &
\overrightarrow{a}_{m-1}\bullet \overrightarrow{a}_{i-1} & 0 &
\overrightarrow{a}_{m-1}\bullet \overrightarrow{a}_{i+1} & \cdots &
\overrightarrow{a}_{m-1}\bullet \overrightarrow{a}_{m} \\
\overrightarrow{0} & \cdots & \overrightarrow{0} & \overrightarrow{a}_{i} &
\overrightarrow{0} & \cdots & \overrightarrow{0}%
\end{array}%
\right\vert .
\end{equation*}

We get, using these notations and rearranging in a suitable manner the terms
of the determinants,%
\begin{equation*}
\overrightarrow{s}=\sum_{i=1}^{m}\alpha _{i}\overrightarrow{a}_{i}=\frac{%
\left\vert
\begin{array}{ccccc}
\overrightarrow{a}_{1}\bullet \overrightarrow{a}_{1} & \overrightarrow{a}%
_{1}\bullet \overrightarrow{a}_{2} & \cdots & \overrightarrow{a}_{1}\bullet
\overrightarrow{a}_{m-1} & \overrightarrow{a}_{1}\bullet \overrightarrow{a}%
_{m} \\
\overrightarrow{a}_{2}\bullet \overrightarrow{a}_{1} & \overrightarrow{a}%
_{2}\bullet \overrightarrow{a}_{2} & \cdots & \overrightarrow{a}_{2}\bullet
\overrightarrow{a}_{m-1} & \overrightarrow{a}_{2}\bullet \overrightarrow{a}%
_{m} \\
\vdots & \vdots & \ddots & \vdots & \vdots \\
\overrightarrow{a}_{m-1}\bullet \overrightarrow{a}_{1} & \overrightarrow{a}%
_{m-1}\bullet \overrightarrow{a}_{2} & \cdots & \overrightarrow{a}%
_{m-1}\bullet \overrightarrow{a}_{m-1} & \overrightarrow{a}_{m-1}\bullet
\overrightarrow{a}_{m} \\
\overrightarrow{a}_{1} & \overrightarrow{a}_{2} & \cdots & \overrightarrow{a}%
_{m-1} & \overrightarrow{a}_{m}%
\end{array}%
\right\vert }{\left\vert
\begin{array}{cccc}
\overrightarrow{a}_{1}\bullet \overrightarrow{a}_{1} & \overrightarrow{a}%
_{1}\bullet \overrightarrow{a}_{2} & \cdots & \overrightarrow{a}_{1}\bullet
\overrightarrow{a}_{m} \\
\overrightarrow{a}_{2}\bullet \overrightarrow{a}_{1} & \overrightarrow{a}%
_{2}\bullet \overrightarrow{a}_{2} & \cdots & \overrightarrow{a}_{2}\bullet
\overrightarrow{a}_{m} \\
\vdots & \vdots & \ddots & \vdots \\
\overrightarrow{a}_{m-1}\bullet \overrightarrow{a}_{1} & \overrightarrow{a}%
_{m-1}\bullet \overrightarrow{a}_{2} & \cdots & \overrightarrow{a}%
_{m-1}\bullet \overrightarrow{a}_{m} \\
\overrightarrow{a}_{m}\bullet \overrightarrow{a}_{1} & \overrightarrow{a}%
_{m}\bullet \overrightarrow{a}_{2} & \cdots & \overrightarrow{a}_{m}\bullet
\overrightarrow{a}_{m}%
\end{array}%
\right\vert }.
\end{equation*}

Part 2. It is just sufficient to use (\ref{eq04}), in order to obtain $%
\left\Vert \overrightarrow{s}\right\Vert ^{2}$.
\end{proof}

For computational purposes, we notice that, in the numerator of (\ref{eq02}%
), the coefficients of the vectors $\overrightarrow{a}_{1},\overrightarrow{a}%
_{2},\ldots ,\overrightarrow{a}_{m}$ are the co-factors of the elements in
the last row of the matrix%
\begin{equation*}
\left[
\begin{array}{ccccc}
\overrightarrow{a}_{1}\bullet \overrightarrow{a}_{1} & \overrightarrow{a}%
_{1}\bullet \overrightarrow{a}_{2} & \cdots & \overrightarrow{a}_{1}\bullet
\overrightarrow{a}_{m-1} & \overrightarrow{a}_{1}\bullet \overrightarrow{a}%
_{m} \\
\overrightarrow{a}_{2}\bullet \overrightarrow{a}_{1} & \overrightarrow{a}%
_{2}\bullet \overrightarrow{a}_{2} & \cdots & \overrightarrow{a}_{2}\bullet
\overrightarrow{a}_{m-1} & \overrightarrow{a}_{2}\bullet \overrightarrow{a}%
_{m} \\
\vdots & \vdots & \ddots & \vdots & \vdots \\
\overrightarrow{a}_{m-1}\bullet \overrightarrow{a}_{1} & \overrightarrow{a}%
_{m-1}\bullet \overrightarrow{a}_{2} & \cdots & \overrightarrow{a}%
_{m-1}\bullet \overrightarrow{a}_{m-1} & \overrightarrow{a}_{m-1}\bullet
\overrightarrow{a}_{m} \\
1 & 1 & \cdots & 1 & 1%
\end{array}%
\right] .
\end{equation*}%
\medskip

\begin{remark}
\textbf{The particular case of Mitrinovic}\newline
The determinantal formula (\ref{eq02}) given in §1 of Proposition 2.1, for
the least norm vector of the given linear variety is a generalization of the
formula of Mitrinovic \cite[page 67]{Mitrinovic} and \cite[page 93]%
{Mitrinovic2}%
\begin{equation*}
x_{k}=\frac{b_{k}\sum\limits_{i=1}^{p}a_{i}^{2}-a_{k}\sum%
\limits_{i=1}^{p}a_{i}b_{i}}{\left( \sum\limits_{i=1}^{p}a_{i}^{2}\right)
\left( \sum\limits_{i=1}^{p}b_{i}^{2}\right) -\left(
\sum\limits_{i=1}^{p}a_{i}b_{i}\right) ^{2}},\qquad k=1,2,\ldots ,p,
\end{equation*}%
where $\left( a_{1},a_{2},\ldots ,a_{p}\right) $ and $\left(
b_{1},b_{2},\ldots ,b_{p}\right) $ are two non proportional sequences of
real numbers satisfying%
\begin{equation*}
\sum\limits_{i=1}^{p}a_{i}x_{i}=0\qquad \text{and}\qquad
\sum\limits_{i=1}^{p}b_{i}x_{i}=1.
\end{equation*}
\end{remark}

\section{The minimum norm vector of a general linear variety}

Here we treat the projection of the origin of the coordinates onto a general
linear variety, so extending Proposition \ref{prop21}. The point where the
sphere centered the origin is tangent to any linear variety is given in
closed form by next relation (\ref{eq10}). This result has been obtained in
a different form and by another approach in \cite{Beesack}.

\begin{theorem}[\protect\cite{Beesack}]
\label{PBeesack}
%Theorem of Beesack (\cite[Theorem 1]{Beesack}, \cite[Theorem 1.7]{Varosanec})}
Let $\overrightarrow{a}_{1},\overrightarrow{a}_{2},\ldots ,\overrightarrow{a}%
_{m}$ be linearly independent vectors in $\mathrm{I\kern-.17emR}^{n}$, with $%
m\geq 2$. The minimum Eu\-cli\-dean norm vector in $\mathrm{I\kern-.17emR}%
^{n}$ satisfying the equations%
\begin{equation}
\begin{array}{l}
\overrightarrow{a}_{1}\bullet \overrightarrow{x}=c_{1} \\
\overrightarrow{a}_{2}\bullet \overrightarrow{x}=c_{2} \\
\text{ \ \ \ \ \ \ \ }\vdots \\
\overrightarrow{a}_{m-1}\bullet \overrightarrow{x}=c_{m-1} \\
\overrightarrow{a}_{m}\bullet \overrightarrow{x}=c_{m},%
\end{array}
\label{eq01}
\end{equation}%
with, at least, one non zero $c_{i}$\emph{, }$\emph{i=1,\ldots ,m}$, is
given by the relation

\begin{equation}
\overrightarrow{s^{\prime }}=\frac{\left\vert
\begin{array}{ccccc}
\overrightarrow{a^{\prime }}_{1}\bullet \overrightarrow{a^{\prime }}_{1} &
\overrightarrow{a^{\prime }}_{1}\bullet \overrightarrow{a^{\prime }}_{2} &
\cdots & \overrightarrow{a^{\prime }}_{1}\bullet \overrightarrow{a^{\prime }}%
_{m-1} & \overrightarrow{a^{\prime }}_{1}\bullet \overrightarrow{a^{\prime }}%
_{m} \\
\vdots & \vdots & \ddots & \vdots & \vdots \\
\overrightarrow{a^{\prime }}_{m-1}\bullet \overrightarrow{a^{\prime }}_{1} &
\overrightarrow{a^{\prime }}_{m-1}\bullet \overrightarrow{a^{\prime }}_{2} &
\cdots & \overrightarrow{a^{\prime }}_{m-1}\bullet \overrightarrow{a^{\prime
}}_{m-1} & \overrightarrow{a^{\prime }}_{m-1}\bullet \overrightarrow{%
a^{\prime }}_{m} \\
\overrightarrow{a^{\prime }}_{1} & \overrightarrow{a^{\prime }}_{2} & \cdots
& \overrightarrow{a^{\prime }}_{m-1} & \overrightarrow{a^{\prime }}_{m}%
\end{array}%
\right\vert }{\left\vert
\begin{array}{ccccc}
\overrightarrow{a^{\prime }}_{1}\bullet \overrightarrow{a^{\prime }}_{1} &
\overrightarrow{a^{\prime }}_{1}\bullet \overrightarrow{a^{\prime }}_{2} &
\cdots & \overrightarrow{a^{\prime }}_{1}\bullet \overrightarrow{a^{\prime }}%
_{m-1} & \overrightarrow{a^{\prime }}_{1}\bullet \overrightarrow{a^{\prime }}%
_{m} \\
\overrightarrow{a^{\prime }}_{2}\bullet \overrightarrow{a^{\prime }}_{1} &
\overrightarrow{a^{\prime }}_{2}\bullet \overrightarrow{a^{\prime }}_{2} &
\cdots & \overrightarrow{a^{\prime }}_{2}\bullet \overrightarrow{a^{\prime }}%
_{m-1} & \overrightarrow{a^{\prime }}_{2}\bullet \overrightarrow{a^{\prime }}%
_{m} \\
\vdots & \vdots & \ddots & \vdots & \vdots \\
\overrightarrow{a^{\prime }}_{m-1}\bullet \overrightarrow{a^{\prime }}_{1} &
\overrightarrow{a^{\prime }}_{m-1}\bullet \overrightarrow{a^{\prime }}_{2} &
\cdots & \overrightarrow{a^{\prime }}_{m-1}\bullet \overrightarrow{a^{\prime
}}_{m-1} & \overrightarrow{a^{\prime }}_{m-1}\bullet \overrightarrow{%
a^{\prime }}_{m} \\
\overrightarrow{a^{\prime }}_{m}\bullet \overrightarrow{a^{\prime }}_{1} &
\overrightarrow{a^{\prime }}_{m}\bullet \overrightarrow{a^{\prime }}_{2} &
\cdots & \overrightarrow{a^{\prime }}_{m}\bullet \overrightarrow{a^{\prime }}%
_{m-1} & \overrightarrow{a^{\prime }}_{m}\bullet \overrightarrow{a^{\prime }}%
_{m}%
\end{array}%
\right\vert },  \label{eq10}
\end{equation}

where
\begin{equation}  \label{ali}
\overrightarrow{a^{\prime }}_{i}=\overrightarrow{a}_{i}-\dfrac{c_{i}}{c_{m}}%
\overrightarrow{a}_{m},\quad i=1,\ldots ,m-1,
\end{equation}%
and
\begin{equation}  \label{alm}
\overrightarrow{a^{\prime }}_{m}=\dfrac{1}{c_{m}}\overrightarrow{a}_{m}.
\end{equation}

Furthermore,

\begin{equation}  \label{eprop31}
\left\Vert \overrightarrow{s^{\prime }}\right\Vert ^{2}=\overrightarrow{%
s^{\prime }}\bullet \overrightarrow{s^{\prime }}=\frac{G\left(
\overrightarrow{a^{\prime }}_{1},\overrightarrow{a^{\prime }}_{2},\ldots ,%
\overrightarrow{a^{\prime }}_{m-1}\right) }{G\left( \overrightarrow{%
a^{\prime }}_{1},\overrightarrow{a^{\prime }}_{2},\ldots ,\overrightarrow{%
a^{\prime }}_{m-1},\overrightarrow{a^{\prime }}_{m}\right) }.
\end{equation}
\end{theorem}

\begin{proof}
Performing elementary matrix operations, we turn into the form $\left[%
\begin{array}{ccccc}
\hspace{-1mm}0\hspace{-1mm} & \hspace{-1mm}0 \hspace{-1mm} & \hspace{-1mm}%
\cdots \hspace{-1mm} & \hspace{-1mm}0\hspace{-1mm} & \hspace{-1mm}1\hspace{%
-1mm}%
\end{array}%
\right] ^{T}$ the last column of the augmented matrix of the system (\ref%
{eq01})%
\begin{equation*}
\left[
\begin{array}{cccccc}
a_{1_{1}} & a_{1_{2}} & \cdots & a_{1_{n-1}} & a_{1_{n}} & c_{1} \\
a_{2_{1}} & a_{2_{2}} & \cdots & a_{2_{n-1}} & a_{2_{n}} & c_{2} \\
\vdots & \vdots & \ddots & \vdots & \vdots & \vdots \\
a_{m-1_{1}} & a_{m-1_{2}} & \cdots & a_{m-1_{n-1}} & a_{m-1_{n}} & c_{m-1}
\\
a_{m_{1}} & a_{m_{2}} & \cdots & a_{m_{n-1}} & a_{m_{n}} & c_{m}%
\end{array}%
\right] ,
\end{equation*}
where $\overrightarrow{a}_{i}=\left(a_{i_{1}},a_{i_{2}},%
\ldots,a_{i_{n-1}},a_{i_{n}}\right)$.
\end{proof}

\section{Projection of a point onto a linear variety}

For dealing with this problem by taking into account the result of the
preceding section, we use the fact that Euclidean distance is preserved
under translations.

We are given a linear variety $V$ and an external point $Q$. We perform a
translation towards the origin $O$ of the coordinates: the pair $(Q,V)$
turns into the pair $(O,V^{\prime })$. We, then, apply Theorem \ref{PBeesack}
to the pair $(O,V^{\prime })$. Finally, we undo the performed translation:
we go back from the origin $O$ to the point $Q$. We state the following

\begin{proposition}
\label{prop41}

Let $\overrightarrow{a}_{1},\overrightarrow{a}_{2},\ldots ,\overrightarrow{a}%
_{m}$ be linearly independent vectors in $\mathrm{I\kern-.17emR}^{n}$, with $%
m\geq 2$. Then:

\begin{enumerate}
\item The projection $S$ of the external point $Q:=\overrightarrow{q}%
=(q_{1},q_{2},\ldots ,q_{n})$ onto the linear variety $V$ defined by
\begin{equation}
\begin{array}{l}
\overrightarrow{a}_{1}\bullet \overrightarrow{x}=c_{1} \\
\overrightarrow{a}_{2}\bullet \overrightarrow{x}=c_{2} \\
\text{ \ \ \ \ \ \ \ }\vdots \\
\overrightarrow{a}_{m-1}\bullet \overrightarrow{x}=c_{m-1} \\
\overrightarrow{a}_{m}\bullet \overrightarrow{x}=c_{m},%
\end{array}
\label{eq41}
\end{equation}%
with, at least, one non zero $c_{i}$\emph{, }${i=1,\ldots ,m}$, is given by
\begin{equation}
S:=\overrightarrow{s}=\overrightarrow{s^{\prime \prime }}+\overrightarrow{q},
\label{psv}
\end{equation}%
where
\begin{equation}
\overrightarrow{s^{\prime \prime }}=\frac{\left\vert
\begin{array}{ccccc}
\overrightarrow{a^{\prime \prime }}_{1}\bullet \overrightarrow{a^{\prime
\prime }}_{1} & \overrightarrow{a^{\prime \prime }}_{1}\bullet
\overrightarrow{a^{\prime \prime }}_{2} & \cdots & \overrightarrow{a^{\prime
\prime }}_{1}\bullet \overrightarrow{a^{\prime \prime }}_{m-1} &
\overrightarrow{a^{\prime \prime }}_{1}\bullet \overrightarrow{a^{\prime
\prime }}_{m} \\
\vdots & \vdots & \ddots & \vdots & \vdots \\
\overrightarrow{a^{\prime \prime }}_{m-1}\bullet \overrightarrow{a^{\prime
\prime }}_{1} & \overrightarrow{a^{\prime \prime }}_{m-1}\bullet
\overrightarrow{a^{\prime \prime }}_{2} & \cdots & \overrightarrow{a^{\prime
\prime }}_{m-1}\bullet \overrightarrow{a^{\prime \prime }}_{m-1} &
\overrightarrow{a^{\prime \prime }}_{m-1}\bullet \overrightarrow{a^{\prime
\prime }}_{m} \\
\overrightarrow{a^{\prime \prime }}_{1} & \overrightarrow{a^{\prime \prime }}%
_{2} & \cdots & \overrightarrow{a^{\prime \prime }}_{m-1} & \overrightarrow{%
a^{\prime \prime }}_{m}%
\end{array}%
\right\vert }{\left\vert
\begin{array}{ccccc}
\overrightarrow{a^{\prime \prime }}_{1}\bullet \overrightarrow{a^{\prime
\prime }}_{1} & \overrightarrow{a^{\prime \prime }}_{1}\bullet
\overrightarrow{a^{\prime \prime }}_{2} & \cdots & \overrightarrow{a^{\prime
\prime }}_{1}\bullet \overrightarrow{a^{\prime \prime }}_{m-1} &
\overrightarrow{a^{\prime \prime }}_{1}\bullet \overrightarrow{a^{\prime
\prime }}_{m} \\
\overrightarrow{a^{\prime \prime }}_{2}\bullet \overrightarrow{a^{\prime
\prime }}_{1} & \overrightarrow{a^{\prime \prime }}_{2}\bullet
\overrightarrow{a^{\prime \prime }}_{2} & \cdots & \overrightarrow{a^{\prime
\prime }}_{2}\bullet \overrightarrow{a^{\prime \prime }}_{m-1} &
\overrightarrow{a^{\prime \prime }}_{2}\bullet \overrightarrow{a^{\prime
\prime }}_{m} \\
\vdots & \vdots & \ddots & \vdots & \vdots \\
\overrightarrow{a^{\prime \prime }}_{m-1}\bullet \overrightarrow{a^{\prime
\prime }}_{1} & \overrightarrow{a^{\prime \prime }}_{m-1}\bullet
\overrightarrow{a^{\prime \prime }}_{2} & \cdots & \overrightarrow{a^{\prime
\prime }}_{m-1}\bullet \overrightarrow{a^{\prime \prime }}_{m-1} &
\overrightarrow{a^{\prime \prime }}_{m-1}\bullet \overrightarrow{a^{\prime
\prime }}_{m} \\
\overrightarrow{a^{\prime \prime }}_{m}\bullet \overrightarrow{a^{\prime
\prime }}_{1} & \overrightarrow{a^{\prime \prime }}_{m}\bullet
\overrightarrow{a^{\prime \prime }}_{2} & \cdots & \overrightarrow{a^{\prime
\prime }}_{m}\bullet \overrightarrow{a^{\prime \prime }}_{m-1} &
\overrightarrow{a^{\prime \prime }}_{m}\bullet \overrightarrow{a^{\prime
\prime }}_{m}%
\end{array}%
\right\vert },  \label{eq13}
\end{equation}%
with%
\begin{equation}
\begin{array}{l}
\overrightarrow{a^{\prime \prime }}_{i}=\overrightarrow{a}_{i}-\dfrac{%
c_{i}^{\prime }}{c_{m}^{\prime }}\overrightarrow{a}_{m},\quad i=1,\ldots
,m-1, \\
\overrightarrow{a^{\prime \prime }}_{m}=\dfrac{1}{c_{m}^{\prime }}%
\overrightarrow{a}_{m}%
\end{array}
\label{aliiq}
\end{equation}%
and
\begin{equation}
c_{i}^{\prime }=c_{i}-\overrightarrow{a}_{i}\bullet \overrightarrow{q}.
\label{clinha}
\end{equation}

\item For the distance, we have
\begin{equation}
d^2(Q,V)=d^2(O,V^{\prime })=\left\Vert \overrightarrow{s^{\prime \prime }}%
\right\Vert ^{2}=\frac{G\left( \overrightarrow{a^{\prime \prime }}_{1},%
\overrightarrow{a^{\prime \prime }}_{2},\ldots ,\overrightarrow{a^{\prime
\prime }}_{m-1}\right) }{G\left( \overrightarrow{a^{\prime \prime }}_{1},%
\overrightarrow{a^{\prime \prime }}_{2},\ldots ,\overrightarrow{a^{\prime
\prime }}_{m-1},\overrightarrow{a^{\prime \prime }}_{m}\right) }.
\end{equation}
\end{enumerate}
\end{proposition}

\begin{proof}
\begin{enumerate}
\item We perform a translation towards the origin of the coordinates, of the
pair $(Q,V)$ in order to get the pair $(O,V^{\prime })$. We have
\begin{equation}
\overrightarrow{x^{\prime }}=\overrightarrow{x}+\overrightarrow{QO}=%
\overrightarrow{x}-\overrightarrow{q}.  \notag
\end{equation}%
Replacing, in equations (\ref{eq41}), $\overrightarrow{x}$ with $%
\overrightarrow{x^{\prime }}+\overrightarrow{q}$, we get
\begin{equation}
\begin{array}{l}
\overrightarrow{a}_{1}\bullet \overrightarrow{x^{\prime }}=c_{1}^{\prime }
\\
\overrightarrow{a}_{2}\bullet \overrightarrow{x^{\prime }}=c_{2}^{\prime }
\\
\text{ \ \ \ \ \ \ \ }\vdots \\
\overrightarrow{a}_{m-1}\bullet \overrightarrow{x^{\prime }}=c_{m-1}^{\prime
} \\
\overrightarrow{a}_{m}\bullet \overrightarrow{x^{\prime }}=c_{m}^{\prime },%
\end{array}
\label{eq42}
\end{equation}%
with, at least, one non-zero $c_{i}^{\prime }$ and $c_{i}^{\prime }=c_{i}-%
\overrightarrow{a_{i}}\bullet \overrightarrow{q}$.

Now, by using relations (\ref{eq10}), (\ref{ali}), (\ref{alm}) and (\ref%
{eprop31}), we obtain the relations (\ref{eq13}), (\ref{aliiq}) and (\ref%
{clinha}).

Finally, undoing the translation, we have
\begin{equation*}
\overrightarrow{s}=\overrightarrow{s^{\prime \prime }}+\overrightarrow{q}.
\end{equation*}

\item The Euclidean distance is translation invariant:
\begin{equation}
d^{2}(Q,V)=d^{2}(Q,S)=d^{2}(O,V^{\prime })=\left\Vert \overrightarrow{%
s^{\prime \prime }}\right\Vert ^{2}=\frac{G\left( \overrightarrow{a^{\prime
\prime }}_{1},\overrightarrow{a^{\prime \prime }}_{2},\ldots ,%
\overrightarrow{a^{\prime \prime }}_{m-1}\right) }{G\left( \overrightarrow{%
a^{\prime \prime }}_{1},\overrightarrow{a^{\prime \prime }}_{2},\ldots ,%
\overrightarrow{a^{\prime \prime }}_{m-1},\overrightarrow{a^{\prime \prime }}%
_{m}\right) }.  \notag
\end{equation}
\end{enumerate}
\end{proof}

\section{Distance between two linear varieties}

In this section we deal with the interesting problem of finding the best
approximation pair of points of two given disjoint and non-parallel linear
varieties $V_{1}$ and $V_{2}$. In other words, we are looking for the point $%
S_{1}$ on the linear variety $V_{1}$ and the point $S_{2}$ on the linear
variety $V_{2}$ such that the vector $\overrightarrow{S_{1}S_{2}}$ is, to
within a signal, the shortest one linking the referred to linear varieties.
Here the main tool is the Proposition \ref{prop41}. This result is applied
twice, just bearing in mind that, in the present case, the external point is
either the generic point $G_{V_{1}}:=\overrightarrow{g_{V_{1}}}$ of the
linear variety $V_{1}$ or the generic point $G_{V_{2}}:=\overrightarrow{%
g_{V_{2}}}$ of the linear variety $V_{2}$.

Some notation is in order, for the sake of simplicity of the statement of
our next result.

We write the vector $\overrightarrow{f}\in \mathrm{I\kern-.17emR}^{n}$ the
following manner:%
\begin{equation*}
\overrightarrow{f}=\left( f_{1},f_{2},\ldots ,f_{h},f_{h+1},f_{h+2},\ldots
,f_{n}\right) :=\left( f_{1},f_{2},\ldots ,f_{h},\overrightarrow{\varphi }%
\right) \in \mathrm{I\kern-.17emR}^{h}\times \mathrm{I\kern-.17emR}^{n-h}.
\end{equation*}

We state the main result of this paper

\begin{proposition}
\label{prop51} Let us consider two disjoint and non-parallel linear
varieties $V_{1}$ and $V_{2}$ given, respectively, by

\begin{equation}
V_{1}:=\left\{
\begin{array}{l}
\overrightarrow{a}_{1}\bullet \overrightarrow{x}=c_{1} \\
\overrightarrow{a}_{2}\bullet \overrightarrow{x}=c_{2} \\
{\hspace*{0.7cm}{\vdots }} \\
\overrightarrow{a}_{m_{1}-1}\bullet \overrightarrow{x}=c_{m_{1}-1} \\
\overrightarrow{a}_{m_{1}}\bullet \overrightarrow{x}=c_{m_{1}},%
\end{array}%
\right.  \label{z1}
\end{equation}%
where $\overrightarrow{a}_{1},\overrightarrow{a}_{2},\ldots ,\overrightarrow{%
a}_{m_{1}}$ are linearly independent vectors in $\mathrm{I\kern-.17emR}^{n}$
and with, at least, one non zero scalar $c_{i},$ $i=1,\ldots ,m_{1}$, $%
m_{1}\geq 2$, and

\begin{equation}
V_{2}:=\left\{
\begin{array}{l}
\overrightarrow{b}_{1}\bullet \overrightarrow{y}=d_{1} \\
\overrightarrow{b}_{2}\bullet \overrightarrow{y}=d_{2} \\
{\hspace*{0.7cm}{\vdots }} \\
\overrightarrow{b}_{m_{2}-1}\bullet \overrightarrow{y}=d_{m_{2}-1} \\
\overrightarrow{b}_{m_{2}}\bullet \overrightarrow{y}=d_{m_{2}},%
\end{array}%
\right.  \label{var2}
\end{equation}%
where $\overrightarrow{b}_{1},\overrightarrow{b}_{2},\ldots ,\overrightarrow{%
b}_{m_{2}}$ are linearly independent vectors in $\mathrm{I\kern-.17emR}^{n}$
and with, at least, one non zero scalar $d_{i},$ $i=1,\ldots ,m_{2}$, $%
m_{2}\geq 2$.

Let us denote $\overrightarrow{x}=\left( x_{1},\ldots
,x_{m_{1}},x_{m_{1}+1},\ldots ,x_{n}\right) \in V_{1}$ as $\overrightarrow{x}%
=\left( x_{1},\ldots ,x_{m_{1}},\overrightarrow{\xi }\right) \in \mathrm{I%
\kern-.17emR}^{m_{1}}\times \mathrm{I\kern-.17emR}^{n-m_{1}}$ and $%
\overrightarrow{y}=\left( y_{1},\ldots ,y_{m_{2}},y_{m_{2}+1},\ldots
,y_{n}\right) \in V_{2}$ as $\overrightarrow{y}=\left( y_{1},\ldots
,y_{m_{2}},\overrightarrow{\eta }\right) \in \mathrm{I\kern-.17emR}%
^{m_{2}}\times \mathrm{I\kern-.17emR}^{n-m_{2}}$.

Let us denote by $\left[ S_{1}S_{2}\right] $ the shortest straight line
segment connecting the two linear varieties $V_{1}$ and $V_{2}$.

Then

\begin{enumerate}
\item The points $S_{1}\in V_{1}$ and $S_{2}\in V_{2}$ are obtained through
the unique solution of the overdetermined consistent system of linear
algebraic equations
\begin{equation}
\displaystyle\left\{
\begin{array}{l}
S_{1}\left( \overrightarrow{\eta }\right) =G_{V_{1}}\left( \overrightarrow{%
\xi }\right) \\
S_{2}\left( \overrightarrow{\xi }\right) =G_{V_{2}}\left( \overrightarrow{%
\eta }\right) \\
\end{array}%
\right.  \label{sistemacons}
\end{equation}%
where:

\begin{enumerate}
\item[(i)] $G_{V_{1}}\left( \overrightarrow{\xi }\right) :=G_{V_{1}}\left(
x_{m_{1}+1},x_{m_{1}+2},\ldots ,x_{n}\right) $ and $G_{V_{2}}\left(
\overrightarrow{\eta }\right) :=G_{V_{2}}\left(
y_{m_{2}+1},y_{m_{2}+2},\ldots ,y_{n}\right) $ are the generic points of,
respectively, the linear varieties $V_{1}$ and $V_{2}$;

\item[(ii)] $S_{1}\left( \overrightarrow{\eta }\right) :=S_{1}\left(
y_{m_{2}+1},y_{m_{2}+2},\ldots ,y_{n}\right) $

and

$S_{2}\left( \overrightarrow{\xi }\right) :=S_{2}\left(
x_{m_{1}+1},x_{m_{1}+2},\ldots ,x_{n}\right) $

are given, respectively, by

$S_{1}\left( \overrightarrow{\eta }\right) =S_{1}^{\prime \prime }\left(
\overrightarrow{\eta }\right) +G_{V_{2}}\left( \overrightarrow{\eta }\right)
$

and

$S_{2}\left( \overrightarrow{\xi }\right) =S_{2}^{\prime }\left(
\overrightarrow{\xi }\right) +G_{V_{1}}\left( \overrightarrow{\xi }\right) ;$

and where:

\item[(iii)] $S_{1}^{\prime \prime }\left( \overrightarrow{\eta }\right) $
is given by
\begin{equation}
\displaystyle\overrightarrow{s_{1}^{\prime \prime }}\left( \overrightarrow{%
\eta }\right) :=S_{1}^{\prime \prime }\left( \overrightarrow{\eta }\right) =%
\dfrac{\left\vert
\begin{array}{ccccc}
\overrightarrow{a^{\prime \prime }}_{1}\bullet \overrightarrow{a^{\prime
\prime }}_{1} & \cdots & \overrightarrow{a^{\prime \prime }}_{1}\bullet
\overrightarrow{a^{\prime \prime }}_{m_{1}-1} & \overrightarrow{a^{\prime
\prime }}_{1}\bullet \overrightarrow{a^{\prime \prime }}_{m_{1}} &  \\
\vdots & \cdots & \vdots & \vdots &  \\
\overrightarrow{a^{\prime \prime }}_{m_{1}-1}\bullet \overrightarrow{%
a^{\prime \prime }}_{1} & \cdots & \overrightarrow{a^{\prime \prime }}%
_{m_{1}-1}\bullet \overrightarrow{a^{\prime \prime }}_{m_{1}-1} &
\overrightarrow{a^{\prime \prime }}_{m_{1}-1}\bullet \overrightarrow{%
a^{\prime \prime }}_{m_{1}} &  \\
\overrightarrow{a^{\prime \prime }}_{1} & \cdots & \overrightarrow{a^{\prime
\prime }}_{m_{1}-1} & \overrightarrow{a^{\prime \prime }}_{m_{1}} &
\end{array}%
\right\vert }{\left\vert
\begin{array}{ccccc}
\overrightarrow{a^{\prime \prime }}_{1}\bullet \overrightarrow{a^{\prime
\prime }}_{1} & \cdots & \overrightarrow{a^{\prime \prime }}_{1}\bullet
\overrightarrow{a^{\prime \prime }}_{m_{1}-1} & \overrightarrow{a^{\prime
\prime }}_{1}\bullet \overrightarrow{a^{\prime \prime }}_{m_{1}} &  \\
\overrightarrow{a^{\prime \prime }}_{2}\bullet \overrightarrow{a^{\prime
\prime }}_{1} & \cdots & \overrightarrow{a^{\prime \prime }}_{2}\bullet
\overrightarrow{a^{\prime \prime }}_{m_{1}-1} & \overrightarrow{a^{\prime
\prime }}_{2}\bullet \overrightarrow{a^{\prime \prime }}_{m_{1}} &  \\
\vdots & \cdots & \vdots & \vdots &  \\
\overrightarrow{a^{\prime \prime }}_{m_{1}-1}\bullet \overrightarrow{%
a^{\prime \prime }}_{1} & \cdots & \overrightarrow{a^{\prime \prime }}%
_{m_{1}-1}\bullet \overrightarrow{a^{\prime \prime }}_{m_{1}-1} &
\overrightarrow{a^{\prime \prime }}_{m_{1}-1}\bullet \overrightarrow{%
a^{\prime \prime }}_{m_{1}} &  \\
\overrightarrow{a^{\prime \prime }}_{m_{1}}\bullet \overrightarrow{a^{\prime
\prime }}_{1} & \cdots & \overrightarrow{a^{\prime \prime }}_{m_{1}}\bullet
\overrightarrow{a^{\prime \prime }}_{m_{1}-1} & \overrightarrow{a^{\prime
\prime }}_{m_{1}}\bullet \overrightarrow{a^{\prime \prime }}_{m_{1}} &
\end{array}%
\right\vert },  \label{s1duaslinhas}
\end{equation}%
being%
\begin{equation}
\displaystyle%
\begin{array}{l}
\displaystyle\overrightarrow{a^{\prime \prime }}_{i}=\overrightarrow{a}_{i}-%
\dfrac{c_{i}^{\prime }}{c_{m_{1}}^{\prime }}\displaystyle\overrightarrow{a}%
_{m_{1}},\quad i=1,\ldots ,m_{1}-1, \\
\displaystyle\overrightarrow{a^{\prime \prime }}_{m_{1}}=\dfrac{1}{%
c_{m_{1}}^{\prime }}\overrightarrow{a}_{m_{1}}%
\end{array}
\notag
\end{equation}%
with, at least, one non zero $c_{i}^{\prime }=c_{i}-\overrightarrow{a}%
_{i}\bullet \overrightarrow{g_{V_{2}}}$, $i=1,\ldots ,m_{1}$,

and
\begin{equation}
\overrightarrow{s_{2}^{\prime \prime }}\left( \overrightarrow{\xi }\right)
:=S_{2}^{\prime \prime }\left( \overrightarrow{\xi }\right) =\frac{%
\left\vert
\begin{array}{ccccc}
\overrightarrow{b^{\prime \prime }}_{1}\bullet \overrightarrow{b^{\prime
\prime }}_{1} & \cdots & \overrightarrow{b^{\prime \prime }}_{1}\bullet
\overrightarrow{b^{\prime \prime }}_{m_{2}-1} & \overrightarrow{b^{\prime
\prime }}_{1}\bullet \overrightarrow{b^{\prime \prime }}_{m_{2}} &  \\
\vdots & \cdots & \vdots & \vdots &  \\
\overrightarrow{b^{\prime \prime }}_{m_{2}-1}\bullet \overrightarrow{%
b^{\prime \prime }}_{1} & \cdots & \overrightarrow{b^{\prime \prime }}%
_{m_{2}-1}\bullet \overrightarrow{b^{\prime \prime }}_{m_{2}-1} &
\overrightarrow{b^{\prime \prime }}_{m_{2}-1}\bullet \overrightarrow{%
b^{\prime \prime }}_{m_{2}} &  \\
\overrightarrow{b^{\prime \prime }}_{1} & \cdots & \overrightarrow{b^{\prime
\prime }}_{m_{2}-1} & \overrightarrow{b^{\prime \prime }}_{m_{2}} &
\end{array}%
\right\vert }{\left\vert
\begin{array}{ccccc}
\overrightarrow{b^{\prime \prime }}_{1}\bullet \overrightarrow{b^{\prime
\prime }}_{1} & \cdots & \overrightarrow{b^{\prime \prime }}_{1}\bullet
\overrightarrow{b^{\prime \prime }}_{m_{2}-1} & \overrightarrow{b^{\prime
\prime }}_{1}\bullet \overrightarrow{b^{\prime \prime }}_{m_{2}} &  \\
\overrightarrow{b^{\prime \prime }}_{2}\bullet \overrightarrow{b^{\prime
\prime }}_{1} & \cdots & \overrightarrow{b^{\prime \prime }}_{2}\bullet
\overrightarrow{b^{\prime \prime }}_{m_{2}-1} & \overrightarrow{b^{\prime
\prime }}_{2}\bullet \overrightarrow{b^{\prime \prime }}_{m_{2}} &  \\
\vdots & \cdots & \vdots & \vdots &  \\
\overrightarrow{b^{\prime \prime }}_{m_{2}-1}\bullet \overrightarrow{%
b^{\prime \prime }}_{1} & \cdots & \overrightarrow{b^{\prime \prime }}%
_{m_{2}-1}\bullet \overrightarrow{b^{\prime \prime }}_{m_{2}-1} &
\overrightarrow{b^{\prime \prime }}_{m_{2}-1}\bullet \overrightarrow{%
b^{\prime \prime }}_{m_{2}} &  \\
\overrightarrow{b^{\prime \prime }}_{m_{2}}\bullet \overrightarrow{b^{\prime
\prime }}_{1} & \cdots & \overrightarrow{b^{\prime \prime }}_{m_{2}}\bullet
\overrightarrow{b^{\prime \prime }}_{m_{2}-1} & \overrightarrow{b^{\prime
\prime }}_{m_{2}}\bullet \overrightarrow{b^{\prime \prime }}_{m_{2}} &
\end{array}%
\right\vert },  \label{s22linhas}
\end{equation}%
being%
\begin{equation}
\displaystyle%
\begin{array}{l}
\displaystyle\overrightarrow{b^{\prime \prime }}_{i}=\overrightarrow{b}_{i}-%
\dfrac{d_{i}^{\prime }}{d_{m_{2}}^{\prime }}\displaystyle\overrightarrow{b}%
_{m_{2}},\quad i=1,\ldots ,m_{2}-1, \\
\displaystyle\overrightarrow{b^{\prime \prime }}_{m_{2}}=\dfrac{1}{%
d_{m_{2}}^{\prime }}\overrightarrow{b}_{m_{2}}%
\end{array}
\notag
\end{equation}
\end{enumerate}

with, at least, one non zero $d_{i}^{\prime }=d_{i}-\overrightarrow{b}%
_{i}\bullet \overrightarrow{g_{V_{1}}}$, $i=1,\ldots ,m_{2}$.

\item The distance $d(V_{1},V_{2})$ between the two linear varieties is
given by
\begin{equation*}
d\left( V_{1},V_{2}\right) =\left\Vert \overrightarrow{S_{1}S_{2}}%
\right\Vert .
\end{equation*}
\end{enumerate}
\end{proposition}

\begin{proof}
Essentially the proof consists on dealing once at a time with the two linear
varieties $V_{1}$ and $V_{2}$:

\begin{enumerate}
\item finding the generic point of each linear variety;

\item applying the Proposition \ref{prop41}.

In the following way:

\item[(i)] The generic points

From the underdetermined system (\ref{z1}), we can, without loss of
generality, assume that the generic point $G_{V_{1}}:=\overrightarrow{%
g_{V_{1}}}$ depends on the $n-m_{1}+1$ parameters $x_{m_{1}+1},x_{m_{1}+2},%
\ldots ,x_{n}$.

We write
\begin{equation}
G_{V_{1}}=G_{V_{1}}\left( \overrightarrow{\xi }\right) =\left[
\begin{array}{c}
x_{1}(x_{m_{1}+1},\ldots ,x_{n}) \\
\vdots \\
x_{m_{1}}(x_{m_{1}+1},\ldots ,x_{n}) \\
x_{m_{1}+1} \\
\vdots \\
x_{n} \\
\end{array}%
\right] :=\overrightarrow{g_{V_{1}}}.
\end{equation}

Similarly, we write for the generic point $G_{V_{2}}:=\overrightarrow{%
g_{V_{2}}}$ of the linear variety $V_{2}$:
\begin{equation}
G_{V_{2}}=G_{V_{2}}\left( \overrightarrow{\eta }\right) =\left[
\begin{array}{c}
y_{1}(y_{m_{2}+1},\ldots ,y_{n}) \\
\vdots \\
y_{m_{2}}(y_{m_{2}+1},\ldots ,y_{n}) \\
y_{m_{2}+1} \\
\vdots \\
y_{n} \\
\end{array}%
\right] :=\overrightarrow{g_{V_{2}}}.
\end{equation}

\item[(ii)] The application of the Proposition \ref{prop41}

\begin{enumerate}
\item Concerning the pair $(G_{V_{2}},V_{1})$, we get
\begin{equation}
\displaystyle S_{1}^{\prime \prime }(y_{m_{2}+1},y_{m_{2}+2},\ldots ,y_{n})=%
\dfrac{\left\vert
\begin{array}{ccccc}
\overrightarrow{a^{\prime \prime }}_{1}\bullet \overrightarrow{a^{\prime
\prime }}_{1} & \cdots & \overrightarrow{a^{\prime \prime }}_{1}\bullet
\overrightarrow{a^{\prime \prime }}_{m_{1}-1} & \overrightarrow{a^{\prime
\prime }}_{1}\bullet \overrightarrow{a^{\prime \prime }}_{m_{1}} &  \\
\vdots & \cdots & \vdots & \vdots &  \\
\overrightarrow{a^{\prime \prime }}_{m_{1}-1}\bullet \overrightarrow{%
a^{\prime \prime }}_{1} & \cdots & \overrightarrow{a^{\prime \prime }}%
_{m_{1}-1}\bullet \overrightarrow{a^{\prime \prime }}_{m_{1}-1} &
\overrightarrow{a^{\prime \prime }}_{m_{1}-1}\bullet \overrightarrow{%
a^{\prime \prime }}_{m_{1}} &  \\
\overrightarrow{a^{\prime \prime }}_{1} & \cdots & \overrightarrow{a^{\prime
\prime }}_{m_{1}-1} & \overrightarrow{a^{\prime \prime }}_{m_{1}} &
\end{array}%
\right\vert }{\left\vert
\begin{array}{ccccc}
\overrightarrow{a^{\prime \prime }}_{1}\bullet \overrightarrow{a^{\prime
\prime }}_{1} & \cdots & \overrightarrow{a^{\prime \prime }}_{1}\bullet
\overrightarrow{a^{\prime \prime }}_{m_{1}-1} & \overrightarrow{a^{\prime
\prime }}_{1}\bullet \overrightarrow{a^{\prime \prime }}_{m_{1}} &  \\
\overrightarrow{a^{\prime \prime }}_{2}\bullet \overrightarrow{a^{\prime
\prime }}_{1} & \cdots & \overrightarrow{a^{\prime \prime }}_{2}\bullet
\overrightarrow{a^{\prime \prime }}_{m_{1}-1} & \overrightarrow{a^{\prime
\prime }}_{2}\bullet \overrightarrow{a^{\prime \prime }}_{m_{1}} &  \\
\vdots & \cdots & \vdots & \vdots &  \\
\overrightarrow{a^{\prime \prime }}_{m_{1}-1}\bullet \overrightarrow{%
a^{\prime \prime }}_{1} & \cdots & \overrightarrow{a^{\prime \prime }}%
_{m_{1}-1}\bullet \overrightarrow{a^{\prime \prime }}_{m_{1}-1} &
\overrightarrow{a^{\prime \prime }}_{m_{1}-1}\bullet \overrightarrow{%
a^{\prime \prime }}_{m_{1}} &  \\
\overrightarrow{a^{\prime \prime }}_{m_{1}}\bullet \overrightarrow{a^{\prime
\prime }}_{1} & \cdots & \overrightarrow{a^{\prime \prime }}_{m_{1}}\bullet
\overrightarrow{a^{\prime \prime }}_{m_{1}-1} & \overrightarrow{a^{\prime
\prime }}_{m_{1}}\bullet \overrightarrow{a^{\prime \prime }}_{m_{1}} &
\end{array}%
\right\vert },  \notag
\end{equation}%
where%
\begin{equation}
\displaystyle%
\begin{array}{l}
\displaystyle\overrightarrow{a^{\prime \prime }}_{i}=\overrightarrow{a}_{i}-%
\dfrac{c_{i}^{\prime }}{c_{m_{1}}^{\prime }}\displaystyle\overrightarrow{a}%
_{m_{1}},\quad i=1,\ldots ,m_{1}-1, \\
\displaystyle\overrightarrow{a^{\prime \prime }}_{m_{1}}=\dfrac{1}{%
c_{m_{1}}^{\prime }}\overrightarrow{a}_{m_{1}}%
\end{array}
\notag
\end{equation}%
with, at least, one non zero $c_{i}^{\prime }=c_{i}-\overrightarrow{a}%
_{i}\bullet \overrightarrow{g_{V_{2}}}$, $i=1,\ldots ,m_{1}$.

\item Concerning the pair $(G_{V_{1}},V_{2})$, we get
\begin{equation}
S_{2}^{\prime \prime }(x_{m_{1}+1},x_{m_{1}+2},\ldots ,x_{n})=\frac{%
\left\vert
\begin{array}{ccccc}
\overrightarrow{b^{\prime \prime }}_{1}\bullet \overrightarrow{b^{\prime
\prime }}_{1} & \cdots & \overrightarrow{b^{\prime \prime }}_{1}\bullet
\overrightarrow{b^{\prime \prime }}_{m_{2}-1} & \overrightarrow{b^{\prime
\prime }}_{1}\bullet \overrightarrow{b^{\prime \prime }}_{m_{2}} &  \\
\vdots & \cdots & \vdots & \vdots &  \\
\overrightarrow{b^{\prime \prime }}_{m_{2}-1}\bullet \overrightarrow{%
b^{\prime \prime }}_{1} & \cdots & \overrightarrow{b^{\prime \prime }}%
_{m_{2}-1}\bullet \overrightarrow{b^{\prime \prime }}_{m_{2}-1} &
\overrightarrow{b^{\prime \prime }}_{m_{2}-1}\bullet \overrightarrow{%
b^{\prime \prime }}_{m_{2}} &  \\
\overrightarrow{b^{\prime \prime }}_{1} & \cdots & \overrightarrow{b^{\prime
\prime }}_{m_{2}-1} & \overrightarrow{b^{\prime \prime }}_{m_{2}} &
\end{array}%
\right\vert }{\left\vert
\begin{array}{ccccc}
\overrightarrow{b^{\prime \prime }}_{1}\bullet \overrightarrow{b^{\prime
\prime }}_{1} & \cdots & \overrightarrow{b^{\prime \prime }}_{1}\bullet
\overrightarrow{b^{\prime \prime }}_{m_{2}-1} & \overrightarrow{b^{\prime
\prime }}_{1}\bullet \overrightarrow{b^{\prime \prime }}_{m_{2}} &  \\
\overrightarrow{b^{\prime \prime }}_{2}\bullet \overrightarrow{b^{\prime
\prime }}_{1} & \cdots & \overrightarrow{b^{\prime \prime }}_{2}\bullet
\overrightarrow{b^{\prime \prime }}_{m_{2}-1} & \overrightarrow{b^{\prime
\prime }}_{2}\bullet \overrightarrow{b^{\prime \prime }}_{m_{2}} &  \\
\vdots & \cdots & \vdots & \vdots &  \\
\overrightarrow{b^{\prime \prime }}_{m_{2}-1}\bullet \overrightarrow{%
b^{\prime \prime }}_{1} & \cdots & \overrightarrow{b^{\prime \prime }}%
_{m_{2}-1}\bullet \overrightarrow{b^{\prime \prime }}_{m_{2}-1} &
\overrightarrow{b^{\prime \prime }}_{m_{2}-1}\bullet \overrightarrow{%
b^{\prime \prime }}_{m_{2}} &  \\
\overrightarrow{b^{\prime \prime }}_{m_{2}}\bullet \overrightarrow{b^{\prime
\prime }}_{1} & \cdots & \overrightarrow{b^{\prime \prime }}_{m_{2}}\bullet
\overrightarrow{b^{\prime \prime }}_{m_{2}-1} & \overrightarrow{b^{\prime
\prime }}_{m_{2}}\bullet \overrightarrow{b^{\prime \prime }}_{m_{2}} &
\end{array}%
\right\vert },  \notag
\end{equation}%
where%
\begin{equation}
\displaystyle%
\begin{array}{l}
\displaystyle\overrightarrow{b^{\prime \prime }}_{i}=\overrightarrow{b}_{i}-%
\dfrac{d_{i}^{\prime }}{d_{m_{2}}^{\prime }}\displaystyle\overrightarrow{b}%
_{m_{2}},\quad i=1,\ldots ,m_{2}-1, \\
\displaystyle\overrightarrow{b^{\prime \prime }}_{m_{2}}=\dfrac{1}{%
d_{m_{2}}^{\prime }}\overrightarrow{b}_{m_{2}}%
\end{array}
\notag
\end{equation}

with, at least, one non zero $d^{\prime }_i=d_i-\overrightarrow{b}%
_{i}\bullet \overrightarrow{g_{V_1}}$, $i=1,\ldots ,m_2$.
\end{enumerate}

Essentially, the points $S_{1}^{\prime \prime }$ and $S_{2}^{\prime \prime }$
resulted from translations of the pairs $(G_{V_{2}},V_{1})$ and $%
(G_{V_{1}},V_{2})$. Undoing the translations, follows
\begin{equation}
\begin{array}{c}
S_{1}=S_{1}^{\prime \prime }+G_{V_{2}} \\
S_{2}=S_{2}^{\prime \prime }+G_{V_{1}}. \\
\end{array}
\notag
\end{equation}

We must get the unique solution of the overdetermined system
\begin{equation}
\displaystyle\left\{
\begin{array}{l}
S_{1}(y_{m_{2}+1},y_{m_{2}+2},\ldots
,y_{n})=G_{V_{1}}(x_{m_{1}+1},x_{m_{1}+2},\ldots ,x_{n}) \\
S_{2}(x_{m_{1}+1},x_{m_{1}+2},\ldots
,x_{n})=G_{V_{2}}(y_{m_{2}+1},y_{m_{2}+2},\ldots ,y_{n}) \\
\end{array}%
\right.  \label{usa}
\end{equation}%
of $2n$ equations and the $(n-m_{1}+1)+(n-m_{2}+1)$ indeterminates
\begin{equation*}
\displaystyle x_{m_{1}+1},x_{m_{1}+2},\ldots
,x_{n},y_{m_{2}+1},y_{m_{2}+2},\ldots ,y_{n}.
\end{equation*}

This system is consistent and has the unique solution
\begin{equation*}
(\displaystyle x_{m_{1}+1}^{\ast },x_{m_{1}+2}^{\ast },\ldots ,x_{n}^{\ast
},y_{m_{2}+1}^{\ast },y_{m_{2}+2}^{\ast },\ldots ,y_{n}^{\ast }).
\end{equation*}%
Hence we obtain
\begin{equation*}
S_{1}=G_{V_{1}}^{\ast }=G_{V_{1}}\left( \overrightarrow{\xi }^{\ast }\right)
=\left[
\begin{array}{c}
x_{1}(x_{m_{1}+1}^{\ast },\ldots ,x_{n}^{\ast }) \\
\vdots \\
x_{m_{1}}(x_{m_{1}+1}^{\ast },\ldots ,x_{n}^{\ast }) \\
x_{m_{1}+1}^{\ast } \\
\vdots \\
x_{n}^{\ast } \\
\end{array}%
\right] =\left[
\begin{array}{c}
x_{1}^{\ast } \\
\vdots \\
x_{m_{1}}^{\ast } \\
x_{m_{1}+1}^{\ast } \\
\vdots \\
x_{n}^{\ast } \\
\end{array}%
\right]
\end{equation*}%
and%
\begin{equation*}
S_{2}=G_{V_{2}}^{\ast }=G_{V_{2}}\left( \overrightarrow{\eta }^{\ast
}\right) =\left[
\begin{array}{c}
y_{1}(y_{m_{2}+1}^{\ast },\ldots ,y_{n}^{\ast }) \\
\vdots \\
y_{m_{2}}(y_{m_{2}+1}^{\ast },\ldots ,y_{n}^{\ast }) \\
y_{m_{2}+1}^{\ast } \\
\vdots \\
y_{n}^{\ast } \\
\end{array}%
\right] =\left[
\begin{array}{c}
y_{1}^{\ast } \\
\vdots \\
y_{m_{2}}^{\ast } \\
y_{m_{2}+1}^{\ast } \\
\vdots \\
y_{n}^{\ast } \\
\end{array}%
\right] .
\end{equation*}

The assertion on consistence of system (\ref{usa}) and uniqueness of the
solution of system (\ref{usa}) is supported by results on existence and
uniqueness of best approximation problems \cite[page 64, Theorem 1]%
{Luenberger} \cite[page 45, Théorème 2.2.5]{Laurent}.
\end{enumerate}
\end{proof}

\begin{remark}
Some attention must be paid to the formulas (\ref{s1duaslinhas}) and (\ref%
{s22linhas}). In fact, we have
\begin{equation}
\overrightarrow{s_{1}^{\prime \prime }}=\overrightarrow{s_{1}^{\prime \prime
}}\left( \overrightarrow{\eta }\right) =\frac{1}{A}\sum_{i=1}^{m_{1}}A_{i}%
\overrightarrow{a_{i}^{\prime \prime }}  \label{s1duaslinhasa}
\end{equation}%
where $A$, $A_{i}$, $i=1,\ldots ,m_{1}$ are higher-degree polynomials in
several variables $y_{m_{2}+1},y_{m_{2}+2},\ldots ,y_{n}$ and
\begin{equation}
\overrightarrow{s_{2}^{\prime \prime }}=\overrightarrow{s_{2}^{\prime \prime
}}\left( \overrightarrow{\xi }\right) =\frac{1}{B}\sum_{j=1}^{m_{2}}B_{j}%
\overrightarrow{b_{j}^{\prime \prime }}  \label{s22linhasa}
\end{equation}%
where $B$, $B_{j}$, $j=1,\ldots ,m_{2}$ are higher-degree polynomials in
several variables $x_{m_{1}+1},x_{m_{2}+2},\ldots ,x_{n}$

However, from (\ref{s1duaslinhasa}) and (\ref{s22linhasa}) we have
\begin{equation}
\overrightarrow{s_{1}^{\prime \prime }}=\overrightarrow{s_{1}^{\prime \prime
}}\left( \overrightarrow{\eta }\right) =\sum_{i=1}^{n}L_{1i}\left(
\overrightarrow{\eta }\right) \overrightarrow{e_{i}}  \label{s1duaslinhasb}
\end{equation}

where $L_{1i}$, $i=1,\ldots ,n$, are first degree polynomials in the
variables $y_{m_{2}+1},y_{m_{2}+2},\ldots ,y_{n}$ and
\begin{equation}
\overrightarrow{s_{2}^{\prime \prime }}=\overrightarrow{s_{2}^{\prime \prime
}}\left( \overrightarrow{\xi }\right) =\sum_{i=1}^{n}L_{2i}\left(
\overrightarrow{\xi }\right) \overrightarrow{e_{i}}  \label{s22linhasb}
\end{equation}%
where $L_{2i}$, $i=1,\ldots ,n$, are first degree polynomials in the
variables $x_{m_{1}+1},x_{m_{2}+2},\ldots ,x_{n}$.

This question is worth a longer explanation. As follows:

By performing the mentioned convenient translations on the systems (\ref{z1}%
) and (\ref{var2}), we obtain two systems where the right hand sides are
vectors whose entries are linear expressions in the parameters that are
coordinates of the vectors $\overrightarrow{G_{V_{1}}}$ and $\overrightarrow{%
G_{V_{2}}}$. By using arguments involving the uniqueness of (least squares)
solution of a linear system by using the Moore-Penrose inverse, we assert
that the solutions of the afore referred to systems are given in terms of
such parameters. The best solution in the least squares sense of the system $%
A\overrightarrow{x}=\overrightarrow{b}$ is given \cite[page 439]{Meyer} by $%
\overrightarrow{x}=A^{\dagger }\overrightarrow{b}$, where $A^{\dagger }$
stands for the Moore-Penrose inverse of matrix $A$. In our case, $A^{\dagger
}$ is a constant matrix, so $\overrightarrow{x}$ depends on the parameters
in vector $\overrightarrow{b}$.

Hence,
\begin{equation*}
\overrightarrow{s_{1}^{\prime \prime }}=\left[
\begin{array}{c}
L_{1}(y_{m_{2}+1},y_{m_{2}+2},\ldots ,y_{n}) \\[5pt]
L_{2}(y_{m_{2}+1},y_{m_{2}+2},\ldots ,y_{n}) \\[5pt]
\vdots \\[5pt]
L_{n}(y_{m_{2}+1},y_{m_{2}+2},\ldots ,y_{n}) \\
\end{array}%
\right]
\end{equation*}%
and
\begin{equation*}
\overrightarrow{s_{2}^{\prime \prime }}=\left[
\begin{array}{c}
L_{1}(x_{m_{1}+1},x_{m_{1}+2},\ldots ,x_{n}) \\[5pt]
L_{2}(x_{m_{1}+1},x_{m_{1}+2},\ldots ,x_{n}) \\[5pt]
\vdots \\[5pt]
L_{n}(x_{m_{1}+1},x_{m_{1}+2},\ldots ,x_{n}) \\
\end{array}%
\right] .
\end{equation*}
\end{remark}

For the sake of clarity, we synthesize:

\begin{itemize}
\item[\textbf{Scholium }] \emph{Regarding the given linear varieties and
without loss of generality, we can write%
\begin{equation*}
V_{1}=\left\{ \left[
\begin{array}{c}
x_{1}(x_{m_{1}+1},\ldots ,x_{n}) \\
\vdots \\
x_{m_{1}}(x_{m_{1}+1},\ldots ,x_{n}) \\
x_{m_{1}+1} \\
\vdots \\
x_{n} \\
\end{array}%
\right] :\left( x_{m_{1}+1},\ldots ,x_{n}\right) \in \mathrm{I\kern-.17emR}%
^{n-m_{1}}\right\}
\end{equation*}%
and%
\begin{equation*}
V_{2}=\left\{ \left[
\begin{array}{c}
y_{1}(y_{m_{2}+1},\ldots ,y_{n}) \\
\vdots \\
y_{m_{2}}(y_{m_{2}+1},\ldots ,y_{n}) \\
y_{m_{2}+1} \\
\vdots \\
y_{n} \\
\end{array}%
\right] :(y_{m_{2}+1},\ldots ,y_{n})\in \mathrm{I\kern-.17emR}%
^{n-m_{2}}\right\} .
\end{equation*}%
Hence we may write
\begin{equation*}
S_{1}=G_{V_{1}}^{\ast }=\left[
\begin{array}{c}
x_{1}(x_{m_{1}+1}^{\ast },\ldots ,x_{n}^{\ast }) \\
\vdots \\
x_{m_{1}}(x_{m_{1}+1}^{\ast },\ldots ,x_{n}^{\ast }) \\
x_{m_{1}+1}^{\ast } \\
\vdots \\
x_{n}^{\ast } \\
\end{array}%
\right] =\left[
\begin{array}{c}
x_{1}^{\ast } \\
\vdots \\
x_{m_{1}}^{\ast } \\
x_{m_{1}+1}^{\ast } \\
\vdots \\
x_{n}^{\ast } \\
\end{array}%
\right]
\end{equation*}%
and%
\begin{equation*}
S_{2}=G_{V_{2}}^{\ast }=\left[
\begin{array}{c}
y_{1}(y_{m_{2}+1}^{\ast },\ldots ,y_{n}^{\ast }) \\
\vdots \\
y_{m_{2}}(y_{m_{2}+1}^{\ast },\ldots ,y_{n}^{\ast }) \\
y_{m_{2}+1}^{\ast } \\
\vdots \\
y_{n}^{\ast } \\
\end{array}%
\right] =\left[
\begin{array}{c}
y_{1}^{\ast } \\
\vdots \\
y_{m_{2}}^{\ast } \\
y_{m_{2}+1}^{\ast } \\
\vdots \\
y_{n}^{\ast } \\
\end{array}%
\right] ,
\end{equation*}%
where
\begin{equation*}
(\displaystyle x_{m_{1}+1}^{\ast },x_{m_{1}+2}^{\ast },\ldots ,x_{n}^{\ast
},y_{m_{2}+1}^{\ast },y_{m_{2}+2}^{\ast },\ldots ,y_{n}^{\ast })
\end{equation*}%
is the unique solution of the overdetermined system (\ref{usa}).}
\end{itemize}

A\ classical projection theorem \cite[page 64, Theorem 1]{Luenberger} \cite[%
page 45, Théorème 2.2.5]{Laurent} \cite[page 64, Exercise 2]{Deutsch}
concerning the case of a point and a linear variety, leads us to a result on
the projection vector connecting two linear varieties. It is a
characterization of the pair of best approximation points, that may be
useful when testing the accuracy of numerical examples.

\begin{proposition}
Let $V_{1}$ and $V_{2}$ be two non-parallel linear varieties: $%
V_{1}=P_{1}+M_{1}$ and $V_{2}=P_{2}+M_{2}$, where $M_{1}$ and $M_{2}$ are
subspaces of $\mathrm{I\kern-.17emR}^{n}$ and $P_{1}$ and $P_{2}$ are fixed
points in $\mathrm{I\kern-.17emR}^{n}$. Then, the unique points $S_{1}\in
V_{1}$ and $S_{2}\in V_{2}$ form a best approximation pair $\left(
S_{1},S_{2}\right) $ of the linear varieties $V_{1}$ and $V_{2}$\ if and
only if the two vectors whose extremities are $S_{1}$ and $S_{2}$ are
orthogonal simultaneously to the subspaces $M_{1}$ and $M_{2}$.
\end{proposition}

\begin{proof}
We need just two facts: the definition of a vector orthogonal to a set of $%
\mathrm{I\kern-.17emR}^{n}$ where a vector is said to be orthogonal to set
if it is orthogonal to each vector of the set; and a projection theorem,
where it is stated that the projection vector is orthogonal to the unique
subspace associated to the given linear variety and not to the linear
variety itself \cite[page 64, Theorem 1]{Luenberger} \cite[page 45, Théorème
2.2.5]{Laurent} \cite[page 64, Exercise 2]{Deutsch}.

We have:

\begin{enumerate}
\item $S_{2}=\overrightarrow{s_{2}}$ is the projection of $S_{1}:=%
\overrightarrow{s_{1}}$ onto the linear variety $V_{2}$: hence $%
\overrightarrow{S_{1}S_{2}}$ is orthogonal to the subspace $M_{2}$;

\item $S_{1}=\overrightarrow{s_{1}}$ is the projection of $S_{2}:=%
\overrightarrow{s_{2}}$ onto the linear variety $V_{1}$: hence $%
\overrightarrow{S_{1}S_{2}}$ is orthogonal to the subspace $M_{1}$.
\end{enumerate}
\end{proof}

Notice that the vector $\overrightarrow{S_{1}S_{2}}$ is not orthogonal
either to the linear varieties $V_{1}$ or $V_{2}$.

Finally, we have a result concerning the separating hyperplanes \cite[pages
105-106]{Deutsch} and the smallest sphere tangent to the two linear
varieties simultaneously.

\begin{itemize}
\item[\textbf{Corollary }] The smallest sphere $S$ tangent to the linear
varieties $V_{1}$ and $V_{2}$ is given by%
\begin{equation*}
S=\left\{ \overrightarrow{x}\in \mathrm{I\kern-.17emR}^{n}:\left\Vert
\overrightarrow{x}-\frac{\overrightarrow{s}_{1}+\overrightarrow{s}_{2}}{2}%
\right\Vert =\left\Vert \frac{\overrightarrow{s}_{1}-\overrightarrow{s}_{2}}{%
2}\right\Vert \right\}
\end{equation*}%
and the supporting hyperplanes are%
\begin{equation*}
H_{i}=\left\{ \overrightarrow{x}\in \mathrm{I\kern-.17emR}^{n}:\left(
\overrightarrow{s}_{1}-\overrightarrow{s}_{2}\right) \bullet \left(
\overrightarrow{x}-\overrightarrow{s_{i}}\right) =0\right\} \text{, \ }i=1,2.
\end{equation*}
\end{itemize}

\section{Illustrative numerical example}

We are given two linear varieties. We exhibit the best two approximation
points --- one point on each linear variety --- and show that the vector $%
\overrightarrow{S_{1}S_{2}}$ is orthogonal to both the subspace $M_{1}$ and
the subspace $M_{2}$ associated to the linear varieties $V_{1}$ and $V_{2}$,
respectively, but not to the linear varieties themselves.

Let the two linear varieties $V_{1}$ and $V_{2}$ be defined as follows
\begin{equation}
V_{1}:=\left\{
\begin{array}{l}
\\overrightarrow{a}_{1}\bullet \overrightarrow{x}=1 \\
\overrightarrow{a}_{2}\bullet \overrightarrow{x}=2,%
\end{array}%
\right.  \label{ex1}
\end{equation}%
with $\overrightarrow{a}_{1}=(1,-1,-2,1,1)$ and $\overrightarrow{a}%
_{2}=(1,1,-4,1,2)$;

\begin{equation}
V_{2}:=\left\{
\begin{array}{l}
\overrightarrow{b}_{1}\bullet \overrightarrow{y}=-10 \\
\overrightarrow{b}_{2}\bullet \overrightarrow{y}=-20 \\
\overrightarrow{b}_{3}\bullet \overrightarrow{y}=3,%
\end{array}%
\right.  \label{var22}
\end{equation}%
with $\overrightarrow{b}_{1}=(1,-1,-2,1,1)$, $\overrightarrow{b}%
_{2}=(-1,1,-4,1,2)$ and $\overrightarrow{b}_{3}=(1,1,-4,-1,3)$.

\begin{itemize}
\item[{\protect\Huge *}] Concerning the Proposition 5.1.

\begin{itemize}
\item[(I)] The generic points $G_{V_{1}}$ and $G_{V_{2}}$ of the linear
varieties $V_{1}$ and $V_{2}$ are%
\begin{equation*}
G_{V_{1}}=\left[
\begin{array}{c}
\frac{3}{2}+3x_{3}-x_{4}-\frac{3}{2}x_{5} \\[5pt]
\frac{1}{2}+x_{3}-\frac{1}{2}x_{5} \\[5pt]
x_{3} \\[5pt]
x_{4} \\[5pt]
x_{5} \\
\end{array}%
\right]
\end{equation*}%
and%
\begin{equation*}
G_{V_{2}}=\left[
\begin{array}{c}
\frac{23}{2}+y_{4}-\frac{1}{2}y_{5} \\[5pt]
\frac{23}{2}+\frac{4}{3}y_{4}-\frac{1}{2}y_{5} \\[5pt]
5+\frac{1}{3}y_{4}+\frac{1}{2}y_{5} \\[5pt]
y_{4} \\[5pt]
y_{5} \\[5pt]
\end{array}%
\right] .
\end{equation*}

\item[(II)] We perform a translation along the vector $\overrightarrow{%
G_{V_{2}}(y_{4},y_{5})O}=O-G_{V_{2}}(y_{4},y_{5})$; the linear variety $%
V_{1}^{\prime }$ is obtained by replacing $\overrightarrow{x}$ in the
relation (\ref{ex1}) with $\overrightarrow{x^{^{\prime }}}+\overrightarrow{%
G_{V_{2}}}$:%
\begin{equation}
V_{1}^{\prime }:=\left\{
\begin{array}{l}
\overrightarrow{a}_{1}\bullet \overrightarrow{x^{\prime }}=11 \\
\overrightarrow{a}_{2}\bullet \overrightarrow{x^{\prime }}=-1-2y_{4}+y_{5}%
\end{array}%
\right. ,
\end{equation}

getting%
\begin{equation}
S_{1}^{\prime }(y_{4},y_{5})=\left[
\begin{array}{c}
\frac{15}{7}+\frac{2}{21}y_{4}-\frac{1}{21}y_{5} \\[5pt]
-\frac{131}{21}-\frac{38}{63}y_{4}+\frac{19}{63}y_{5} \\[5pt]
-\frac{4}{21}+\frac{20}{63}y_{4}-\frac{10}{63}y_{5} \\[5pt]
\frac{15}{7}+\frac{2}{21}y_{4}-\frac{1}{21}y_{5} \\[5pt]
\frac{2}{21}-\frac{10}{63}y_{4}+\frac{5}{63}y_{5} \\[5pt]
\end{array}%
\right]  \notag
\end{equation}%
and%
\begin{equation}
S_{1}(y_{4},y_{5})=S_{1}^{\prime }+G_{V_{2}}=\left[
\begin{array}{c}
\frac{191}{14}+\frac{23}{21}y_{4}-\frac{23}{42}y_{5} \\[5pt]
\frac{221}{42}+\frac{46}{63}y_{4}-\frac{25}{126}y_{5} \\[5pt]
\frac{101}{21}+\frac{41}{63}y_{4}+\frac{43}{126}y_{5} \\[5pt]
\frac{15}{7}+\frac{23}{21}y_{4}-\frac{1}{21}y_{5} \\[5pt]
\frac{2}{21}-\frac{10}{63}y_{4}+\frac{68}{63}y_{5} \\[5pt]
\end{array}%
\right] .  \notag
\end{equation}

Mutatis mutandis:

\item[(III)] We perform a translation along the vector $\overrightarrow{%
G_{V_{1}}(x_{3},x_{4},x_{5})O}=O-G_{V_{1}}(x_{3},x_{4},x_{5})$; the linear
variety $V_{2}^{\prime }$ is obtained by replacing $\overrightarrow{y}$ in
the relation (\ref{var22}) with $\overrightarrow{y^{^{\prime }}}+%
\overrightarrow{G_{V_{1}}}$:%
\begin{equation}
V_{2}^{\prime }:=\left\{
\begin{array}{l}
\overrightarrow{b}_{1}\bullet \overrightarrow{y^{\prime }}=-11 \\
\overrightarrow{b}_{2}\bullet \overrightarrow{y^{\prime }}%
=-19+6x_{3}-2x_{4}-3x_{5} \\
\overrightarrow{b}_{3}\bullet \overrightarrow{y^{\prime }}=1+2x_{4}-x_{5}%
\end{array}%
\right. ,
\end{equation}

getting%
\begin{equation}
S_{2}^{\prime }(x_{3},x_{4},x_{5})=\left[
\begin{array}{c}
\frac{633}{209}-\frac{366}{209}x_{3}+\frac{192}{209}x_{4}+\frac{148}{209}%
x_{5} \\[5pt]
\frac{35}{19}+\frac{12}{19}x_{3}-\frac{1}{19}x_{4}-\frac{15}{38}x_{5} \\%
[5pt]
\frac{674}{209}-\frac{150}{209}x_{3}+\frac{41}{209}x_{4}+\frac{159}{418}x_{5}
\\[5pt]
-\frac{1371}{209}+\frac{240}{209}x_{3}-\frac{191}{209}x_{4}-\frac{129}{418}%
x_{5} \\[5pt]
\frac{172}{209}-\frac{42}{209}x_{3}+\frac{70}{209}x_{4}-\frac{7}{209}x_{5} \\%
[5pt]
\end{array}%
\right]  \notag
\end{equation}%
and%
\begin{equation}
S_{2}(x_{3},x_{4},x_{5})=S_{2}^{\prime }+G_{V_{1}}=\left[
\begin{array}{c}
\frac{1893}{418}+\frac{261}{209}x_{3}-\frac{17}{209}x_{4}-\frac{331}{418}%
x_{5} \\[5pt]
\frac{89}{38}+\frac{31}{19}x_{3}-\frac{1}{19}x_{4}-\frac{17}{19}x_{5} \\%
[5pt]
\frac{674}{209}+\frac{59}{209}x_{3}+\frac{41}{209}x_{4}+\frac{159}{418}x_{5}
\\[5pt]
-\frac{1371}{209}+\frac{240}{209}x_{3}+\frac{18}{209}x_{4}-\frac{129}{418}%
x_{5} \\[5pt]
\frac{172}{209}-\frac{42}{209}x_{3}+\frac{70}{209}x_{4}+\frac{202}{209}x_{5}
\\[5pt]
\end{array}%
\right] .  \notag
\end{equation}

\item[(IV)] Solving the system%
\begin{equation}
\displaystyle\left\{
\begin{array}{l}
S_{1}(y_{4},y_{5})=G_{V_{1}}(x_{3},x_{4},x_{5}) \\
S_{2}(x_{3},x_{4},x_{5})=G_{V_{2}}(y_{4},y_{5}) \\
\end{array}%
\right. ,  \notag
\end{equation}%
we obtain $\displaystyle x_{3}^{\ast }=\frac{837}{848},$ $x_{4}^{\ast }=-%
\frac{4765}{848},$ $x_{5}^{\ast }=\frac{1489}{424},$ $y_{4}^{\ast }=-\frac{%
3560}{509},$ $y_{5}^{\ast }=\frac{453}{212}$.

Hence, using
\begin{equation*}
S_{1}=G_{V_{1}}^{\ast }=\left[
\begin{array}{c}
\frac{3}{2}+3x_{3}^{\ast }-x_{4}^{\ast }-\frac{3}{2}x_{5}^{\ast } \\[5pt]
\frac{1}{2}+x_{3}^{\ast }-\frac{1}{2}x_{5}^{\ast } \\[5pt]
x_{3}^{\ast } \\[5pt]
x_{4}^{\ast } \\[5pt]
x_{5}^{\ast } \\
\end{array}%
\right]
\end{equation*}%
and%
\begin{equation*}
S_{2}=G_{V_{2}}^{\ast }=\left[
\begin{array}{c}
\frac{23}{2}+y_{4}^{\ast }-\frac{1}{2}y_{5}^{\ast } \\[5pt]
\frac{23}{2}+\frac{4}{3}y_{4}^{\ast }-\frac{1}{2}y_{5}^{\ast } \\[5pt]
5+\frac{1}{3}y_{4}^{\ast }+\frac{1}{2}y_{5}^{\ast } \\[5pt]
y_{4}^{\ast } \\[5pt]
y_{5}^{\ast } \\[5pt]
\end{array}%
\right] ,
\end{equation*}%
we, finally, obtain%
\begin{equation*}
S_{1}=\displaystyle\left[
\begin{array}{c}
\frac{77}{16} \\[5pt]
-\frac{57}{212} \\[5pt]
\frac{837}{848} \\[5pt]
-\frac{4765}{848} \\[5pt]
\frac{1489}{424} \\
\end{array}%
\right] \text{ \ and \ }S_{2}=\left[
\begin{array}{c}
\frac{55}{16} \\[5pt]
\frac{469}{424} \\[5pt]
\frac{3169}{848} \\[5pt]
-\frac{3560}{509} \\[5pt]
\frac{453}{212}%
\end{array}%
\right] .
\end{equation*}%
The distance between the two varieties is given by
\begin{equation*}
d\left( V_{1},V_{2}\right) =\Vert \overrightarrow{S_{1}S_{2}}\Vert =\frac{%
2174}{559}.
\end{equation*}
\end{itemize}

\item[{\protect\Huge **}] Concerning the Proposition 5.2.

Let us consider
\begin{equation*}
V_{1}=P_{1}+M_{1}:=\left[
\begin{array}{c}
\frac{3}{2} \\[5pt]
\frac{1}{2} \\[5pt]
0 \\[5pt]
0 \\[5pt]
0 \\
\end{array}%
\right] +\left\{ \left[
\begin{array}{c}
3x_{3}-x_{4}-\frac{3}{2}x_{5} \\[5pt]
x_{3}-\frac{1}{2}x_{5} \\[5pt]
x_{3} \\[5pt]
x_{4} \\[5pt]
x_{5} \\
\end{array}%
\right] :x_{3},x_{4},x_{5}\in \mathrm{I\kern-.17emR}\right\}
\end{equation*}%
and%
\begin{equation*}
V_{2}=P_{2}+M_{2}:=\left[
\begin{array}{c}
\frac{23}{2} \\[5pt]
\frac{23}{2} \\[5pt]
5 \\[5pt]
0 \\[5pt]
0 \\[5pt]
\end{array}%
\right] +\left\{ \left[
\begin{array}{c}
y_{4}-\frac{1}{2}y_{5} \\[5pt]
\frac{4}{3}y_{4}-\frac{1}{2}y_{5} \\[5pt]
\frac{1}{3}y_{4}+\frac{1}{2}y_{5} \\[5pt]
y_{4} \\[5pt]
y_{5} \\[5pt]
\end{array}%
\right] :y_{4},y_{5}\in \mathrm{I\kern-.17emR}\right\} .
\end{equation*}

\begin{itemize}
\item[$\protect\alpha _{1})$] The vector $\overrightarrow{S_{1}S_{2}}$ is
orthogonal to the unique subspace $M_{1}$\ associated to the linear variety $%
V_{1}$. Consider the arbitrarily fixed vector
\begin{equation*}
\overrightarrow{v_{1}}=\left[
\begin{array}{c}
3x_{3}-x_{4}-\frac{3}{2}x_{5} \\[5pt]
x_{3}-\frac{1}{2}x_{5} \\[5pt]
x_{3} \\[5pt]
x_{4} \\[5pt]
x_{5} \\
\end{array}%
\right] \in M_{1}.
\end{equation*}%
We have%
\begin{equation*}
\overrightarrow{S_{1}S_{2}}\cdot \overrightarrow{v_{1}}=0.
\end{equation*}

\item[$\protect\alpha _{2})$] The vector $\overrightarrow{S_{1}S_{2}}$ is
orthogonal to the unique subspace $M_{2}$\ associated to the linear variety $%
V_{2}$. Consider the arbitrarily fixed vector
\begin{equation*}
\overrightarrow{v_{2}}=\left[
\begin{array}{c}
y_{4}-\frac{1}{2}y_{5} \\[5pt]
\frac{4}{3}y_{4}-\frac{1}{2}y_{5} \\[5pt]
\frac{1}{3}y_{4}+\frac{1}{2}y_{5} \\[5pt]
y_{4} \\[5pt]
y_{5} \\[5pt]
\end{array}%
\right] \in M_{2}.
\end{equation*}%
We have%
\begin{equation*}
\overrightarrow{S_{1}S_{2}}\cdot \overrightarrow{v_{2}}=0.
\end{equation*}

\item[$\protect\beta _{1})$] The vector $\overrightarrow{S_{1}S_{2}}$ is not
orthogonal to the linear variety $V_{1}$.

Take the fixed vector%
\begin{equation*}
\overrightarrow{u_{1}}=\left[
\begin{array}{c}
-2 \\[5pt]
0 \\[5pt]
1 \\[5pt]
2 \\[5pt]
3 \\
\end{array}%
\right] \in V_{1}.
\end{equation*}%
We have $\overrightarrow{S_{1}S_{2}}\bullet \overrightarrow{u_{1}}%
=-1.3750\neq 0.$

\item[$\protect\beta _{2})$] The vector $\overrightarrow{S_{1}S_{2}}$ is not
orthogonal to the linear variety $V_{2}$.

Take the fixed vector%
\begin{equation*}
\overrightarrow{u_{2}}=\left[
\begin{array}{c}
\frac{23}{2} \\[5pt]
\frac{71}{6} \\[5pt]
\frac{24}{3} \\[5pt]
1 \\[5pt]
2 \\[5pt]
\end{array}%
\right] \in V_{2}.
\end{equation*}%
We have $\overrightarrow{S_{1}S_{2}}\bullet \overrightarrow{u_{2}}%
=13.7500\neq 0.$
\end{itemize}
\end{itemize}

\section{Conclusions}

In this paper we presented a determinantal formula for the point satisfying
the equality condition in an inequality by Fan and Todd (we answered the
implicit old open question in \cite[page 63]{FanTodd}: to get a closed form
for the minimum norm vector of the given linear variety). In a previous
paper \cite{Vitoria}, we got, by using the center of convenient
hyperquadrics, the point where the inequality (\ref{eq05}) turns into the
equality (\ref{eq04}).

Here, we also restated a determinantal formula for the point of tangency
between a sphere and any linear variety.

Furthermore, we obtained the projection of an external point onto a linear
variety as a quotient of two determinants. Subsequently and consequently
this result was extended for getting the best approximation pair of two
disjoint and non parallel linear varieties. A characterization of this pair
of best approximation points is offered.

\end{document}